\documentclass[a4paper]{article}

\usepackage[english]{babel}
\usepackage[utf8x]{inputenc}
\usepackage[T1]{fontenc}
\usepackage[a4paper,top=3cm,bottom=2cm,left=3cm,right=3cm,marginparwidth=1.75cm]{geometry}
\usepackage{amsmath}
\usepackage{graphicx}
\usepackage[colorinlistoftodos]{todonotes}
\usepackage{listings}
\usepackage{color} %red, green, blue, yellow, cyan, magenta, black, white
\definecolor{mygreen}{RGB}{28,172,0} % color values Red, Green, Blue
\definecolor{mylilas}{RGB}{170,55,241}
\usepackage{changepage}
\usepackage{graphicx}
\usepackage[small]{caption}
\usepackage{subcaption}
\usepackage[makeroom]{cancel}
\usepackage{listings}
\usepackage{amsmath}
\usepackage{amsthm}
\usepackage{amsfonts}
\usepackage{float}
\usepackage{listings}
\usepackage{framed}
\usepackage{cite}
\usepackage{algorithm}
\usepackage{pdfpages}
\usepackage{mathtools}
\usepackage{bm}
\usepackage{authblk}
\usepackage{multirow}
\usepackage{authblk}
\usepackage{tabularx}
\usepackage{pbox}

\usepackage{todonotes}

\DeclareMathOperator{\arctantwo}{arctan2}
\newcommand{\norm}[1]{\left\lVert#1\right\rVert}

\newtheorem{theorem}{Theorem}

\newtheorem{remark}{Remark}

\providecommand{\keywords}[1]{\textit{Keywords and phrases: } #1}

\begin{document}
	
\title{A micro-macro Markov chain Monte Carlo method for molecular dynamics using reaction coordinate proposals I: direct reconstruction}

\author[1]{Hannes Vandecasteele}
%\ead{hannes.vandecasteele@cs.kuleuven.be}
\author[1]{Giovanni Samaey}
%\ead{giovanni.samaey@cs.kuleuven.be}
\affil[1]{KU Leuven, Department of Computer Science, NUMA Section, Celestijnenlaan 200A box 2402, 3001 Leuven, Belgium}
\date{\today}
	
\maketitle

\begin{abstract}
We introduce a new micro-macro Markov chain Monte Carlo method (mM-MCMC) to sample invariant distributions of molecular dynamics systems that exhibit a time-scale separation between the microscopic (fast) dynamics, and the macroscopic (slow) dynamics of some low-dimensional set of reaction coordinates. The algorithm enhances exploration of the state space in the presence of metastability by allowing larger proposal moves at the macroscopic level, on which a conditional accept-reject procedure is applied. Only when the macroscopic proposal is accepted, the full microscopic state is reconstructed from the newly sampled reaction coordinate value and is subjected to a second accept/reject procedure. The computational gain stems from the fact that most proposals are rejected at the macroscopic level, at low computational cost, while microscopic states, once reconstructed, are almost always accepted. We analytically show convergence and discuss the rate of convergence of the proposed algorithm, and numerically illustrate its efficiency on a number of standard molecular test cases. We also investigate the effect of the choice of different numerical parameters on the efficiency of the resulting mM-MCMC method.
\end{abstract}


\section{Introduction} \label{sec:introduction}
Countless systems in chemistry and physics consist of a large number of microscopic particles, of which all positions are collected in the system state $x\in\mathbb{R}^d$, with $d$ the (high) dimension of the system~\cite{leimkuhler2016molecular}. The dynamics of such systems is usually governed by a potential energy $V(x)$ and Brownian motion $W_t$, for instance through the overdamped Langevin dynamics
\begin{equation}\label{eq:overdamped_langevin}
	dX_t = -\nabla V(X_t)dt+\sqrt{2\beta^{-1}}dW_t,
\end{equation} 
in which $X_t$ represents the time-dependent positions of an individual realisation of the dynamics, and $\beta$ is the inverse temperature. Examples of these systems include macromolecules such as polymers~\cite{jourdain2003mathematical}, fluids and solids~\cite{steinhauser2017computational}, and tumour growth~\cite{rejniak2011hybrid}. 

In molecular dynamics, one often wants to sample the time-invariant distribution of such a system, which is the Gibbs-measure
\begin{equation} \label{eq:mu}
d\mu(x) = Z_V^{-1} \exp\left(-\beta V(x)\right)  dx,
\end{equation}
with $Z_V$ the normalization constant and $dx$ the Lebesgue measure. 
Sampling Gibbs distributions is usually achieved via Markov-chain Monte Carlo (MCMC) methods, in which proposal moves, for instance based on the dynamics~\eqref{eq:overdamped_langevin}, are supplemented with an accept-reject criterion. MCMC methods were introduced in by Metropolis in 1953~\cite{metropolis1953equation} %to sample general probability distributions where the normalization constant needs not be known. 
and later generalised by Hastings 
\cite{hastings1970monte}.

Standard MCMC methods can face several computational problems. First, there often exists a large time-scale separation between the fast dynamics of the full, high-dimensional (microscopic) system and the slow behaviour of some suitable low-dimensional (macroscopic) degrees of freedom. Then, for stability reasons, simulating the microscopic dynamics~\eqref{eq:overdamped_langevin} requires taking time steps on the order of the fastest mode of the system, limiting the size of proposal moves and slowing down exploration of the full state space. In particular, when the potential $V$ contains multiple local minima, standard MCMC methods can remain stuck for a long time in these minima. This phenomenon is called metastability. 
There exist several techniques to accelerate sampling in such a context, for instance the parallel replica dynamics~\cite{voter1998parallel,voter2002extending,le2012mathematical,lelievre2016partial}, the adaptive multilevel splitting method~\cite{cerou2007adaptive} and kinetic Monte Carlo~\cite{voter2007introduction}. Second, for high-dimensional problems, simply generating an MCMC proposal may already require a considerable computational effort. 
When the acceptance rate is low, a lot of this computational effort is wasted on proposals that will afterwards be rejected. 
To increase the acceptance rate for high-dimensional problems, one can use modified Gaussian proposals~\cite{durmus2017fast,cotter2013mcmc,abdulle2019accelerated} or add additional (biasing) terms to the potential~\cite{wang2001efficient,henin2004overcoming,darve2001calculating, stoltz2010free}.

For molecular dynamics simulations, quite some effort has been done in obtaining coarse-grained descriptions of the system in terms of a small number of slow degrees of freedom that capture some essential macroscopic features of the system. An important technique is the kinetic Monte Carlo method~\cite{voter2007introduction}, where the macroscopic variables are basins of attraction around local minima of the potential energy. Another example is the adaptive resolution technique~\cite{praprotnik2005adaptive} which models certain regions with the most dynamics with the accurate microscopic model, while the other regions can be simulated accurately with a macroscopic models. Other coarse-graining techniques for molecular dynamics consist of obtaining a macroscopic Brownian dynamics of the microscopic system~\cite{erban2014molecular}, or by simply averaging out all the fast microscopic degrees of freedom to obtain an approximate macroscopic dynamics~\cite{hartmann2007model,pavliotis2008multiscale}. In this manuscript, we use a coarse-graining technique based on \emph{reaction coordinates}. A reaction coordinate is a smooth function from the high-dimensional configuration space $\mathbb{R}^d$ to a lower dimensional space $\mathbb{R}^n$ with $n \ll d$~\cite{stoltz2010free, legoll2010effective}. We denote this function as
%by $\xi$ and a specific value is denoted by $z$,
\begin{equation}
\xi: \mathbb{R}^d \to \mathbb{R}^n, \ x \mapsto \xi(x) = z.
\end{equation} 
Based on the underlying evolution of the molecular system, one can approximate the dynamics of the reaction coordinate values by an \emph{effective dynamics}~\cite{legoll2010effective} of the form
\begin{equation}\label{eq:effdyn_intro}
dZ_t = b(Z_t) d + \sqrt{2\beta^{-1}} \sigma(Z_t) dW_t.
\end{equation}
We discuss the derivation of a suitable effective dynamics for our setting in Section~\ref{sec:effdyn}.

In this manuscript, we propose a new MCMC method, called micro-macro MCMC (mM-MCMC), that aims at exploiting approximate coarse-grained descriptions of the type~\eqref{eq:effdyn_intro} to accelerate sampling of the invariant measure of high-dimensional stochastic processes~\eqref{eq:overdamped_langevin} in presence of a time-scale separation. The objective of the method is to obtain a significantly lower variance on the reaction coordinates than the standard MCMC method, for a given run-time. The mM-MCMC scheme consists of three steps to generate a new sample from the Gibbs measure~\eqref{eq:mu}: (i) \textit{restriction}, i.e., computation of the reaction coordinate value $z$ associated to the current microscopic sample $x$; (ii) a \textit{Macroscopic MCMC step}, i.e., sampling a new value of the reaction coordinate based on the effective dynamics~\eqref{eq:effdyn_intro}; (iii) \emph{reconstruction}, i.e., creation of a microscopic sample based on the sampled reaction coordinate value. Step (ii) contains an accept/reject step at the reaction coordinate level. If the proposal is rejected, we propose a new reaction coordinate value and we only proceed when the macroscopic proposal is accepted. After reconstruction, we perform an additional accept/reject procedure to ensure the exact target distribution~\eqref{eq:mu} is sampled consistently. The mM-MCMC algorithm is discussed in detail in Section~\ref{sec:mM-MCMC}. In Section~\ref{sec:convergence}, we show analytically that the proposed mM-MCMC method samples the correct invariant measure, regardless of the effective dynamics that was used to generate the proposals at the reaction coordinate level. 

The advantage of using an effective dynamics to generate macroscopic proposals for the reaction coordinates crucially depends on the quality of the proposal moves at the reaction coordinate level, for two reasons. First, one needs to ensure that the fastest modes are not present at the reaction coordinate level, such that larger moves are possible than at the microscopic level, enhancing the exploration of the phase space. Second, the scheme should be constructed such that most rejected proposals are already rejected at the reaction coordinate level, i.e., without ever having to perform the (costly) reconstruction of the corresponding microscopic sample. In particular, the acceptance rate of the reconstructed microscopic samples should be close to $1$. If the effective dynamics used to generate reaction coordinate proposals is not an accurate approximation of the exact time-dependent evolution of the reaction coordinate values, more proposals will only be rejected after reconstruction at the microscopic level, which leads to waste of computational efforts. We present numerical results in Section~\ref{sec:results}, in which we also study the effect of the choice of some components in the mM-MCMC method, on the computational efficiency.

The idea of using an effective dynamics to generate coarse-grained proposals was already proposed in the Coupled Coarse Graining MCMC method, introduced in~\cite{kalligiannaki2012coupled,kalligiannaki2012multilevel}, where large lattice systems with an Ising-type potential energy were sampled. In this setting, there are natural expressions for the reconstruction step. In our work, we significantly extend the applicability of such an approach by introducing the use of reaction coordinates to generate microscopic samples that correspond to a given value of the reaction coordinate. 
Similarly, a two-level MCMC algorithm is also used in~\cite{efendiev2006preconditioning} as a `pre-conditioner' to increase the microscopic acceptance rate for fluid flows. Also here, reconstruction is performed in a particular setting, in casu for fluid flows.
Also in other contexts, multilevel MCMC approaches have already been proposed for specific applications. For data assimilation, a multilevel MCMC method was used to sample the posterior distribution in Bayesian inference~\cite{marzouk2013bayesian}. Here, the different levels correspond to a different resolution of a forward PDE evaluation, not to different levels of modelling.  

The remainder of this manuscript is organised as follows. In Section~\ref{sec:effdyn}, we briefly introduce the time-invariant distribution, free energy and effective dynamics of a reaction coordinate. Section~\ref{sec:mM-MCMC} introduces the mM-MCMC method, explaining each algorithmic step in detail. %We explain the mM-MCMC scheme with indirect reconstruction in Section~\ref{subsec:mMMCMCrcs}. In Section~\ref{sec:practical}, we discuss how we use the indirect reconstruction scheme of Section~\ref{subsec:mMMCMCrcs} to efficiently pre-compute an approximation to the free energy and the effective dynamics.  
In Section~\ref{sec:convergence}, we state and prove the convergence result of mM-MCMC method, along with a result that relates the rate of convergence to equilibrium of mM-MCMC to the rate of convergence of the macroscopic MCMC method. In Section~\ref{sec:results}, we apply the mM-MCMC scheme to two molecular dynamics cases: an academic three-atom molecule and butane. In each example, there is a time-scale present between parts of the molecule and we show numerically that mM-MCMC is able to bridge a large part of the time-scale separation. Here, we also illustrate the impact of different choices for the effective dynamics, the approximate macroscopic invariant distribution and the type of reconstruction on the efficiency of the mM-MCMC method over the microscopic MALA method. We discuss in detail the impact of each combination of the above parameters on the macroscopic and microscopic acceptance rate, the runtime and the variance of an estimated quantity of interest. We conclude this manuscript with a summarising discussion and some pointers to future research in Section~\ref{sec:conclusion}. In particular, when the reaction coordinate function has a complicated form, sampling a microscopic sample on the sub-manifold of constant reaction coordinate during for the reconstruction step can be expensive and cumbersome. Therefore, we will introduce an \emph{indirect reconstruction} scheme for the mM-MCMC method in a companion paper~\cite{vandecasteele2020indirect}, to make the reconstruction step more general and efficient. Correspondingly, we will refer to the method in this paper as the mM-MCMC with direct reconstruction.

\section{Reaction coordinates and effective dynamics} \label{sec:effdyn}

In this section, we first introduce the concept of reaction coordinates and give their time-invariant distribution based on the concept of free energy (Section~\ref{subsec:cg}). We then describe the effective dynamics~\cite{legoll2010effective} to obtain an approximate dynamics at the reaction coordinate level (Section~\ref{subsec:effdyn}). Finally, in Section~\ref{subsec:reconstr}, we give a time-invariant \emph{reconstruction distribution} of microscopic samples, given a fixed value of the reaction coordinate, which will be important in the following sections.

\subsection{Coarse-grained descriptions and reaction coordinates} \label{subsec:cg}
A reaction coordinate is a continuous function $\xi$ from the high-dimensional configuration space $\mathbb{R}^d$ to a lower dimensional space $\mathbb{R}^n$ of partial information with $n \ll d$~\cite{stoltz2010free} :
\begin{equation} \label{eq:rc}
\xi : \mathbb{R}^d \to \mathbb{R}^n.
\end{equation}
We denote by $H \subset \mathbb{R}^n$ the image of $\xi$ and by $\Sigma(z)$ the level set of $\xi$ at constant value $z$. Throughout the text, we will use the letter $z$ to denote a value of a reaction coordinate. A useful formula that relates integrals over the set $\mathbb{R}^d$ to integrals over the level sets of the reaction coordinate $\xi$ is the co-area formula~\cite{stoltz2010free}. For any smooth function $f: \mathbb{R}^d \to \mathbb{R}$ we can write
\begin{equation} \label{eq:coarea}
\int_{\mathbb{R}^d} f(x) dx = \int_{H} \int_{\Sigma(z)} f(x) \left(\det G(x)\right)^{-1/2} d\sigma_{\Sigma(z)}(x) \ dz =
\int_{H} \int_{\Sigma(z)} f(x) \ \delta_{\xi(x)-z}(dx) dz, 
\end{equation}
where $d\sigma_{\Sigma(z)}$ is the Lebesgue measure on $\Sigma(z)$, induced by the Lebesgue measure on the ambient space $\mathbb{R}^d$, the Gram matrix $G(x)$ is defined as
\[
G(x) = \nabla \xi(x) ^T \nabla \xi(x),
\]
and $\delta_{\xi(x)-z}(dx) = \left(\det G(x)\right)^{-1/2} d\sigma_{\Sigma(z)}(x)$.

Given the invariant measure $\mu(x)$ for the full microscopic system, we can define the marginal invariant distribution with respect to the reaction coordinates as~\cite{stoltz2010free}
\begin{equation} \label{eq:invcoarse}
\mu_0(z) \propto \int_{\Sigma(z)} \mu(x) \ \delta_{\xi(x)-z}(dx) = \int_{\Sigma(z)} \mu(x) \ \left(\det G(x)\right)^{-1/2} d\sigma_{\Sigma(z)}(x).
\end{equation}
We can associate the marginal distribution $\mu_0$ to the free energy of the reaction coordinate of the system. Using the co-area formula~\eqref{eq:coarea}, we can define the free energy, or the potential energy of the reaction coordinate, by integrating the Gibbs measure on the level sets of $\xi$:
\begin{equation} \label{eq:freeenergy}
A(z) = -\frac{1}{\beta} \ln \left( \int_{\Sigma(z)} Z_V^{-1} \exp(-\beta V(x))  \left(\det G(x)\right)^{-1/2} d \sigma_{\Sigma(z)}(x)\right).
\end{equation}
Putting the definitions of the invariant distribution of the reaction coordinates~\eqref{eq:invcoarse}, the co-area formula~\eqref{eq:coarea} and the free energy~\eqref{eq:freeenergy} together, we can given an alternative expression for $\mu_0(z)$ as
\begin{equation} \label{eq:invariantrc}
\mu_0(z) = Z_A^{-1} \exp(-\beta A(z)),
\end{equation}
where $Z_A$ is the normalization constant. The time-invariant distribution of the reaction coordinate values hence has the same form as the Gibbs measure~\eqref{eq:mu} where the free energy $A(z)$ takes over the role of the potential energy $V(x)$.

\subsection{Effective dynamics: evolution of the reaction coordinates} \label{subsec:effdyn}
As for the overdamped Langevin equation~\eqref{eq:overdamped_langevin} that models the evolution of the microscopic samples in the configuration space, one can also define an evolution equation for the reaction coordinate values~\cite{legoll2010effective}. The exact time-evolution of the reaction coordinate values is given by the stochastic differential equation (SDE)
\begin{equation} \label{eq:exactrcdynamics}
d\tilde{Z}_t = \tilde{b}(\tilde{Z}_t, t) dt + \sqrt{2\beta^{-1}} \tilde{\sigma}(\tilde{Z}_t, t) dW_t,
\end{equation}
where the drift and diffusion coefficients $\tilde{b}$ and $\tilde{\sigma}$ read
\[
\begin{aligned}
\tilde{b}(z, t) &= \mathbb{E}_{\psi(t)}[ \  -\nabla V(X_t) \cdot \nabla \xi(X_t) + \beta^{-1} \triangle V(X_t) \ | \xi(X_t) = z ] \\
\tilde{\sigma}^2(z, t) &= \mathbb{E}_{\psi(t)}[ \  \norm{\nabla \xi(X_t) }^2 \ | \ \xi(X_t) = z ],
\end{aligned}
\]
with $\psi(t)$ the distribution of $X_t$~\eqref{eq:overdamped_langevin}. 

Equation~\eqref{eq:exactrcdynamics} is not closed, since the coefficients $\tilde{b}(z, t)$ and $\tilde{\sigma}^2(z, t)$ are based on the time-dependent distribution $\psi(t)$. One therefore usually considers an approximate equation that %is easier to compute with and 
has the same marginal time-invariant distribution for the reaction coordinates as the exact dynamics~\cite{legoll2010effective}. This dynamics is called the \emph{effective dynamics}, and is defined by 
\begin{equation} \label{eq:effdyn}
dZ_t = b(Z_t) dt + \sqrt{2 \beta^{-1}} \sigma(Z_t) dW_t,
\end{equation}
where the drift and diffusion terms are computed using the Gibbs measure $\mu$,
\begin{equation} \label{eq:effdyncoef}
\begin{aligned}
b(z) &= \mathbb{E}_{\mu}[ \  -\nabla V \cdot \nabla \xi + \beta^{-1} \triangle V \ | \ \xi(X)=z] \\
\sigma^2(z) &= \mathbb{E}_{\mu}[ \ \norm{\nabla \xi }^2 \ | \ \xi(X)=z].
\end{aligned}
\end{equation}
In practice, the coefficients $b(z)$ and $\sigma^2(z)$ need only be pre-computed once on a grid of $z-$values, and we use linear interpolation to compute the coefficients in an in-between reaction coordinate value.

A few numerical schemes have been proposed to compute the free energy and the coefficients $b$ and $\sigma$ in the effective dynamics for a given value of $z$. We mention here a projection dynamics~\cite{stoltz2010free} and a hybrid Monte Carlo method~\cite{lelievre2018hybrid}. As we will explain in the next section, the mM-MCMC scheme allows working with an approximation to the invariant distribution of the reaction coordinate values that can contain significant discretization errors. 

\subsection{Reconstructing microscopic samples from a reaction coordinate value\label{subsec:reconstr}}
To reconstruct a microscopic sample from a reaction coordinate value $z$, we define a reconstruction distribution on the level set $\Sigma(z)$ of microscopic samples $x$ with $\xi(x)=z$.  Consider a probability distribution $\psi(x) dx$ on the microscopic state space $\mathbb{R}^d$. By the co-area formula~\eqref{eq:coarea}, the corresponding probability distribution of $x$ defined on the sub-manifold $\Sigma(z)$ of a constant value of $\xi(x)=z$ reads
\begin{equation*}
d\nu_\psi(x|z) = \frac{\psi(x) \left(\det G(x)\right)^{-1/2} d\sigma_{\Sigma(z)}(x)}{\int_{\Sigma(z)}  \psi(x) \left(\det G(x)\right)^{-1/2} d\sigma_{\Sigma(z)}(x)}.
\end{equation*}
In particular, when $\psi(x)dx$ is the Gibbs measure $\mu(x)dx$, we can define the \emph{exact} time-invariant reconstruction distribution for a reaction coordinate as
\begin{equation} \label{eq:conditionalxgivenz}
d \nu(x|z) = \frac{\exp(-\beta V(x)) \left(\det G(x)\right)^{-1/2} d\sigma_{\Sigma(z)}(x)}{\int_{\Sigma(z)} \exp(-\beta V(x)) \left(\det G(x)\right)^{-1/2}  d\sigma_{\Sigma(z)}(x)}.
\end{equation}

Note that the expression~\eqref{eq:freeenergy} for the free energy $A(z)$ is related to the definition of the reconstruction distribution~\eqref{eq:conditionalxgivenz}, which is the invariant distribution of $x$, conditioned on a value of the reaction coordinate $z$~\eqref{eq:conditionalxgivenz}. In fact, one can give an alternative expression for $\nu$ using the rules of conditional probability
\begin{equation} \label{eq:exactreconstruction2}
d \nu(x|z) = \frac{\mu(x)}{\mu_0(z)} \delta_{\xi(x)-z}(dx)= \frac{Z_A}{Z_V} \frac{\exp(-\beta \ V(x))}{\exp(-\beta \ A(z))} \delta_{\xi(x)-z}(dx),
\end{equation}
which is identical to~\eqref{eq:conditionalxgivenz} by the co-area formula. For a microscopic sample $x$ with reaction coordinate value $z$, we can thus relate the time-invariant densities as $\mu(x) = \nu(x|z) \mu_0(z)$ on $\mathbb{R}^d$.

\section{Micro-macro Markov chain Monte Carlo method} \label{sec:mM-MCMC}
All concepts are now in place to state the general micro-macro Markov chain Monte Carlo (mM-MCMC) algorithm. We assume that there is a `natural' reaction coordinate available in the molecular system, such as an angle or bond length. We present the complete mM-MCMC method with direct reconstruction in this section, and discuss its convergence and rate of convergence properties in Section~\ref{sec:convergence}.

The aim of the mM-MCMC method is to generate a sample from the microscopic Gibbs measure $\mu(x) dx$ on the high-dimensional space $\mathbb{R}^d$. The method relies on the availability of two ingredients. First, we assume that we can sample an approximation $\bar{\mu}_0(z)$ to the exact invariant probability measure $\mu_0$ of the reaction coordinates, using an MCMC method with a macroscopic transition distribution $q_0(\cdot | \cdot)$. This macroscopic sampling is discussed in Section~\ref{subsubsec:coarseprop}.  Second, we require a reconstruction distribution $\bar{\nu}(x|z)$ of microscopic samples conditioned upon a given reaction coordinate value. In principle, the choice of $\bar{\mu}_0$ and $\bar{\nu}$ is arbitrary for the mM-MCMC method to converge. However, these choices influence the efficiency of the resulting method. The reconstruction step is discussed in Section~\ref{subsubsec:reconstruction}. Both steps involve an accept/reject procedure. The complete algorithm is shown in Algorithm~\ref{algo:mM-MCMC}.

\subsection{Generating a macroscopic proposal} \label{subsubsec:coarseprop}
Suppose we start with a microscopic sample $x_n$ that constitutes a sample of $\mu$. To generate a macroscopic proposal according to the approximate distribution $\bar{\mu}_0$, we first restrict the current microscopic sample to its reaction coordinate value, i.e., we compute $z_n=\xi(x_n)$. Next, we propose a new reaction coordinate value $z'$ using the macroscopic transition probability $q_0(z' | z_n)$. This transition kernel can, for instance, be based on the effective dynamics~\eqref{eq:effdyn}, a gradient descent method based on an approximation of the free energy~\eqref{eq:freeenergy}, or even a simple Brownian motion. To ensure that $z'$ samples the prescribed distribution $\bar{\mu}_0$ of the reaction coordinate values, we accept $z'$ with probability
\begin{equation} \label{eq:coarseaccept}
\alpha_{CG}(z' | z_n) = \min\left\{1, \frac{\bar{\mu}_0(z') \ q_0(z_n | z')}{\bar{\mu}_0(z_n) \ q_0(z' | z_n)} \right\},
\end{equation}
which is the standard Metropolis-Hastings form for the acceptance rate.
Therefore, we can define the macroscopic transition kernel $\mathcal{D}$ as
\begin{equation} \label{eq:coarsetransitionkernel}
\mathcal{D}(z' | z_n) = \alpha_{CG}(z' | z_n) \ q_0(z' | z_n) + \left(1 - \int_{H} \alpha_{CG}(y|z_n) \ q_0(y|z_n) \ dy \right) \delta(z'-z_n).
\end{equation}
Note that $\bar{\mu}_0$ is indeed the stationary probability measure associated with $\mathcal{D}$. If $z'$ is accepted, we proceed to the reconstruction step. If not, we return to the first step and define $x_{n+1} = x_n$.

\subsection{Reconstructing a microscopic sample} \label{subsubsec:reconstruction}
If the reaction coordinate value $z'$ has been accepted, we construct a microscopic sample $x'$ by taking one sample from the given reconstruction distribution $\bar{\nu}(\cdot | z')$. Afterwards, we decide on the acceptance of $x'$ in a final accept/reject step. 

To compute the corresponding microscopic acceptance probability, we first define the transition probability distribution on the microscopic level. Starting from the previous microscopic sample $x_n$, the microscopic transition distribution reads
\begin{equation} \label{eq:coarsetransitionprobability}
q(x' | x_n) = \bar{\nu}(x' |  \xi(x')) \ \mathcal{D}(\xi(x') |  \xi(x_n)),
\end{equation}
i.e., the probability of transitioning from $x_n$ to $x'$ is given by the probability of generating and accepting a reaction coordinate value $\xi(x')$, multiplied by the probability of reconstructing the microscopic sample $x'$, given its reaction coordinate value.

 Using the detailed balance condition on the macroscopic level, $\mathcal{D}(z'|  z_n) \ \bar{\mu}_0(z_n) = \mathcal{D}(z_n|  z') \ \bar{\mu}_0(z')$, the  acceptance probability is
%begin{equation}
\begin{align}
\alpha_F(x' |x_n) &= \min \left\{1, \frac{\mu(x') \ q(x_n|  x')}{\mu(x_n) \ q(x'|  x_n)} \right\} \nonumber \\ %%%%%%%%
&= \min \left\{1, \frac{\mu(x') \ \bar{\nu} (x| z_n ) \ \mathcal{D}(z_n| z')}{\mu(x_n) \ \bar{\nu} (x'| z' ) \ \mathcal{D}(z'| z_n)} \right\}  \nonumber \\ %%%%%%%%
&= \min \left\{1, \frac{\mu(x') \ \bar{\mu}_0(z_n) \ \bar{\nu}(x_n |  z_n) }{\mu(x_n) \ \bar{\mu}_0(z') \ \bar{\nu}(x' |  z')  } \right\}. \label{eq:fineacceptreject}
\end{align}
%\end{equation}
%To conclude, after reconstructing $x'$ from $z'$ we compute the acceptance rate $\alpha_F(x' | \ z')$. 
On acceptance, we set $x_{n+1} = x'$. If the microscopic sample is rejected, $x_{n+1} = x_n$.

\begin{remark} \label{rem:exactreconstruction}
There is a special situation in which the microscopic acceptance probability is always $1$. We call this situation `exact reconstruction' and this holds when one can write
\begin{equation} \label{eq:exactreconstruction5}
\mu(x) = \bar{\nu}(x |  \xi(x)) \ \bar{\mu}_0(\xi(x)).
\end{equation}
Since the microscopic distribution $\mu(x)$ can only be decomposed uniquely as $\mu(x) = \nu(x |  \xi(x)) \ \mu_0(\xi(x))$, we must have $\bar{\mu}_0(\xi(x)) = \mu_0(\xi(x))$ and $\bar{\nu}(x |  \xi(x)) = \nu(x |  \xi(x))$ for exact reconstruction.
When~\eqref{eq:exactreconstruction5} holds, it is easy to see that the reconstruction acceptance rate will always be $1$, so that the mM-MCMC  performs no unnecessary computational work during reconstruction. Hence, for computational efficiency, it is beneficial to choose the approximate macroscopic distribution $\bar{\mu}_0$ and the reconstruction distribution $\bar{\nu}$ in such way that the exact reconstruction property~\eqref{eq:exactreconstruction5} approximately holds.
\end{remark}

\begin{remark}
	The micro-macro Markov chain Monte Carlo algorithm can, in principle, be used with other coarse-graining schemes than reaction coordinates, as done in~\cite{kalligiannaki2012coupled,kalligiannaki2012multilevel}. For example, the kinetic Monte Carlo method~\cite{voter2007introduction} defines discrete macroscopic states as regions around the local minima in the potential energy $V(x)$. Sampling these macroscopic states then consists of sampling transition probabilities between the local minima. During reconstruction, we then construct a microscopic sample in the basin of attraction around the given local minimum. The formulation of Algorithm~\ref{algo:mM-MCMC} remains unaltered in such a situation.
\end{remark}

\subsection{The complete algorithm}

The complete mM-MCMC algorithm is shown in Algorithm~\ref{algo:mM-MCMC}.

\begin{algorithm}
\begin{flushleft}
Given a microscopic sample $x_n, \ n = 1, 2, \dots$ .

\vspace{4mm}
(i) \textbf{Restriction}: compute the reaction coordinate value $z_n = \xi(x_n)$.

\vspace{4mm}
(ii) \textbf{Macroscopic Proposal}: 
\begin{itemize}
\item	
Generate a reaction coordinate value $z' \sim q_0 ( \cdot | z_n)$.
\item
Accept the reaction coordinate value with probability
\[
\alpha_{CG}(z' |  z_n) = \min \left\{1, \frac{\bar{\mu}_0(z') \ q_0(z_n | z')}{\bar{\mu}_0(z_n) \ q_0(z' | z_n)} \right\}
\]
\item On acceptance, proceed to step (iii), otherwise set $x_{n+1} = x_n$ and repeat step (ii).
\end{itemize}

\vspace{4mm}
(iii) \textbf{Reconstruction}: 
\begin{itemize}
\item Generate a microscopic sample $x' \sim \bar{\nu}(\cdot | z')$.
\item	Accept the microscopic sample with probability
\[
\alpha_F(x'|x_n) = \min \left\{1, \frac{\mu(x') \ \bar{\mu}_0(z_n) \ \bar{\nu}(x_{n} | z_n) }{\mu(x_n) \ \bar{\mu}_0(z') \ \bar{\nu}(x' | z')  } \right\}.
\]
\item On acceptance, set $x_{n+1} = x'$ and return to step (i) for the next microscopic sample. Otherwise, set $x_{n+1} = x_n$ and generate a new reaction coordinate value in step (ii).
\end{itemize}
\end{flushleft}
\caption{The micro-macro Markov chain Monte Carlo method.}
\label{algo:mM-MCMC}
\end{algorithm}

 \section{Convergence and rate of convergence of mM-MCMC} \label{sec:convergence}
 In this section, we show that the mM-MCMC method with direct reconstruction converges to the correct microscopic invariant distribution $\mu$ and is ergodic under some mild assumptions on the macroscopic transition distribution $q_0(\cdot|\cdot)$, the approximate macroscopic invariant distribution $\bar{\mu}_0$ and the reconstruction distribution $\nu(\cdot|\cdot)$. Furthermore, we show that in case of exact reconstruction~\eqref{eq:exactreconstruction5}, mM-MCMC converges at the same rate to $\mu$ as the macroscopic MCMC sampler converges to its invariant distribution $\mu_0$ of reaction coordinate values. In Section~\ref{subsec:convergencemMMCMC}, we give an expression for the microscopic transition kernel of the mM-MCMC method, and we state and prove the convergence and ergodicity result. Afterwards, in Section~\ref{subsec:ergodicity}, we relate the rate of convergence of mM-MCMC to the microscopic invariant distribution to the rate of convergence of the corresponding macroscopic sampler, in case of exact reconstruction.
 
 \subsection{Convergence of mM-MCMC} \label{subsec:convergencemMMCMC}
 Before formulating the convergence statement, we give an expression for the transition kernel of the mM-MCMC method with direct reconstruction. The probability of transitioning from state $x$ to state $x'$ reads
\[
\mathcal{K}_{mM}(x'| x) = \begin{cases*} \begin{aligned}
\alpha_F(x' |x) \ \bar{\nu}(x' |  \xi(x')) \ \alpha_{CG}(\xi(x')|  \xi(x)) \ q_0(\xi(x')|  \xi(x)) \ \ \ x' \neq x 
\\1 - \int_{\mathbb{R}^d} \alpha_F(y | x) \ \bar{\nu}(y |  \xi(y)) \  \alpha_{CG}(\xi(y)|  \xi(x)) \ q_0(\xi(y)|  \xi(x)) dy \ \ \ x' = x. \\
\end{aligned} \end{cases*}
\]
Using the definition of the microscopic transition probability $q(x'|x)$
\begin{equation} \label{eq:q}
q(x'|  x) = \bar{\nu}(x'|  \xi(x')) \ \alpha_{CG}(\xi(x')|  \xi(x)) \ q_0(\xi(x')| \xi(x)),
\end{equation}
the full micro-macro transition kernel can be written as
\begin{equation} \label{eq:mMtransition kernel}
\begin{aligned}
\mathcal{K}_{mM}(x' | x) 
 &= \alpha_F(x' |x) \ \bar{\nu}(x' |  \xi(x')) \ \alpha_{CG}(\xi(x')|  \xi(x)) \ q_0(\xi(x')|  \xi(x))  \\ &+ \left(1 - \int_{\mathbb{R}^d} \alpha_F(y | x) \ \bar{\nu}(y |  \xi(y)) \  \alpha_{CG}(\xi(y)|  \xi(x)) \ q_0(\xi(y)|  \xi(x)) dy\right) \delta (x' - x) \\
 &=
 \alpha_F(x' | x) \ q(x' |  x) + \left(1 - \int_{\mathbb{R}^d} \alpha_F(y | x) \ q(y |  x) dy\right) \delta (x' - x), \\
\end{aligned}
\end{equation}

We then have the following Theorem:
\begin{theorem} \label{thm:convergence_mM-MCMC}
For every macroscopic transition distribution $q_0$ that is not identical to the exact, time-discrete, transition distribution of the effective dynamics~\eqref{eq:effdyn}, for every approximate macroscopic distribution $\bar{\mu}_0$ with $\text{supp}(\bar{\mu}_0) = H$ and every reconstruction distribution $\bar{\nu}$ such that $q(x'|  x)  > 0. \ \forall x, x' \in \mathbb{R}^d$, 

\begin{itemize}
	\item[(i)] the transition probability kernel~\eqref{eq:mMtransition kernel} satisfies the detailed balance condition with target measure $\mu$;
	
	\item[(ii)] the target measure $\mu$ is a stationary distribution of $\mathcal{K}_{mM}$;
	
	\item[(iii)] the chain $\{x_n\}$ is $\mu-$irreducible;
	
	\item[(iv)] the chain $\{x_n\}$ is aperiodic.
\end{itemize}
\end{theorem}
Let us clarify the statement of Theorem~\ref{thm:convergence_mM-MCMC}. The second statement (ii) implies that the Markov chain has the target distribution $\mu$ as invariant measure, which is naturally a consequence of the detailed balance condition (i). The third and fourth statement ensure convergence and ergodicity of the Markov chain. Ergodicity means that the averages over one sample path, $N^{-1} \sum_{n=1}^N g(x_n)$ converge to averages over the stationary distribution $\int g \ d\mu$ almost surely as $N$ increases to infinity, for every $g \in L_1(\mathbb{R}^d)$. 

A proof similar to that of Theorem~\ref{thm:convergence_mM-MCMC} was already given in~\cite{kalligiannaki2012multilevel} in the specific context of stochastic models defined on lattice systems. Here, we extend that proof to molecular systems with reaction coordinates.

\begin{proof}[\textbf{Proof}]\mbox{}
	
	(i) The case $x' = x$ is trivial. Take  $x \neq x'$ and write for the transition kernel~\eqref{eq:mMtransition kernel}
	\begin{equation*}
	\begin{aligned}
	\mathcal{K}_{mM}(x'|x) \ \mu(x) &= \alpha_F(x' |  x) \ \bar{\nu}(x' |  \xi(x')) \ \alpha_{CG}(\xi(x') |  \xi(x)) \ q_0(\xi(x') |  \xi(x)) \ \mu(x) \\
	&= \min\left\{1, \ \frac{\mu(x') \ \bar{\mu}_0(\xi(x)) \ \bar{\nu}(x|  \xi(x))}{\mu(x) \ \bar{\mu}_0(\xi(x')) \ \bar{\nu}(x'| \xi(x'))}\right\}   \bar{\nu}(x' |  \xi(x'))  \\ &\times \min\left\{1, \ \frac{\bar{\mu}_0(\xi(x')) \ q_0(\xi(x) |  \xi(x')) }{\bar{\mu}_0(\xi(x)) \ q_0(\xi(x') |  \xi(x))} \right\}   q_0(\xi(x') |  \xi(x)) \ \mu(x) \\
	&= \min\left\{\ \mu(x) \ \bar{\mu}_0(\xi(x')) \ \bar{\nu}(x' |  \xi(x')),  \ \mu(x') \ \bar{\mu}_0(\xi(x)) \ \bar{\nu}(x|  \xi(x)) \right\} \\
	&\times \min\left\{ \frac{q_0(\xi(x') |  \xi(x))}{\bar{\mu}_0(\xi(x'))},  \frac{q_0(\xi(x)|  \xi(x'))}{ \bar{\mu}_0(\xi(x))} \right\}  \\
	&= \mathcal{K}_{mM}(x |  x') \ \mu(x'),
	\end{aligned}
	\end{equation*}
	since the third equality is symmetric in $x$ and $x'$.
	
	\vspace{1mm}
	(ii) Follows directly from (i).
	
	\vspace{1mm}
	(iii) To prove that the chain $\{x_n\}$ is $\mu$-irreducible, we need to show that $\mathcal{K}_{mM}(A |  x) > 0$ for all $x \in  \mathbb{R}^d$ and for all measurable sets $A \subset \mathbb{R}^d$ with $\mu(A) > 0$. Note that
	\begin{equation}
	\begin{aligned}
	\mathcal{K}_{mM}(A |  x) &= \int_A \mathcal{K}_{mM}(x' |  x) dx' \geq \int_{A \backslash \{x\}} \mathcal{K}_{mM}(x' |  x) dx' \\
	&= \int_{A \backslash \{x\}} \alpha_F(x'| x) \ \bar{\nu}(x' |  \xi(x')) \ \alpha_{CG}(\xi(x')| \xi(x)) \ q_0(\xi(x')|  \xi(x)) dx'.
	\end{aligned}
	\end{equation}
	The final three factors form $q(x'|x)$~\eqref{eq:q} which is strictly positive by assumption. Since $A \subset \text{supp}(\mu)$, the acceptance rate $\alpha_F$ is positive as well, proving that $\mathcal{K}(A|  x) > 0$.
	
	\vspace{1mm}
	(iv) For $\{x_n\}$ to be aperiodic, it is sufficient to show that there exists an $x \in \text{supp}(\mu)$ such that $\mathcal{K}(\{x\}| x) > 0$, implying that $x_{n+1} = x_n$ can occur with positive probability~\cite{kalligiannaki2012multilevel}.  We will prove this by contradiction. The transition kernel reads
	\[
	\mathcal{K}(\{x\}|  x) = 1 - \int_{\mathbb{R}^d} \alpha_F(x'|x) \ \bar{\nu}(x'|  \xi(x')) \ \alpha_{CG}(\xi(x')| \xi(x)) \ q_0(\xi(x') | \xi(x)) dx'.
	\]
	If $\mathcal{K}(\{x\}|  x) = 0$ for all $x \in \mathbb{R}^d$ then
	\begin{equation} \label{eq:Kequal0}
	\int_{\mathbb{R}^d} \alpha_F(x'|x) \ \bar{\nu}(x'| \xi(x')) \ \alpha_{CG}(\xi(x')| \xi(x)) \ q_0(\xi(x') | \xi(x)) dx' = 1,
	\end{equation}
	implying both acceptance probabilities $\alpha_F(x'|x) $ and $\alpha_{CG}(\xi(x')|  \xi(x))$ should be 1 almost everywhere, because  $q(x' |  x) > 0$ everywhere by assumption. This implies that the proposal kernel $\bar{\nu}(x'|  \xi(x')) \ q_0(\xi(x')|  \xi(x))$ samples from the correct invariant distribution $\mu$ without rejections, which is not the case because the macroscopic transition distribution $q_0$ is not identical to the exact, time-discrete, transition distribution of the effective dynamics~\eqref{eq:effdyn}. Hence, there exists some $x \in \mathbb{R}^d$ such that $\mathcal{K}(\{x\}|  x) > 0$.
\end{proof}

\subsection{Rate of convergence of mM-MCMC in case of exact reconstruction} \label{subsec:ergodicity}

Besides convergence and ergodicity of the mM-MCMC scheme, we can also relate the rate of convergence of the complete mM-MCMC scheme to rate of convergence of the macroscopic MCMC sampler.
In contrast to Theorem~\ref{thm:convergence_mM-MCMC}, we assume that $H$ is compact and that the exact reconstruction property~\eqref{eq:exactreconstruction2} holds. We are not aware whether a similar result holds when the exact reconstruction property does not hold or when $H$ is not compact. The proof relies on a  expression for the mM-MCMC transition kernel that we derive in Appendix~\ref{app:relationKandD}. The remainder of the proof is a straightforward calculation.

\begin{theorem} \label{thm:ergodicity_mM-MCMC}
Assume that the exact reconstruction property~\eqref{eq:exactreconstruction5} holds, that $\nu(\cdot|  z)$ is bounded from above uniformly for any reaction coordinate value $z$ and that the image of the reaction coordinate $H$ is compact. Then, there exist positive constants $\eta$ and $ \kappa < 1$ such that
\begin{equation} \label{eq:ergodic_bound_TV}
\norm{\mathcal{K}_{mM}^n(\cdot|x)-\mu(\cdot)}_{TV} \leq \norm{\mathcal{D}^n(\cdot|\xi(x)) - \mu_0(\cdot)}_{TV} + \eta \ \kappa^n,
\end{equation}
for all $x \in \mathbb{R}^d$. Here, $\mathcal{K}_{mM}^n(x'|x)$ is the probability~\eqref{eq:mMtransition kernel} of reaching the microscopic sample $x'$ after $n$ mM-MCMC steps with initial value $x$. Similarly, $\mathcal{D}^n(\xi(x')|\xi(x))$ is the probability~\eqref{eq:coarsetransitionkernel} of reaching reaction coordinate value $\xi(x')$ after $n$ steps of the macroscopic MCMC sampler with initial value $\xi(x)$.
\end{theorem}

\begin{proof}[\textbf{Proof}] \mbox{} 
	In Appendix~\ref{app:relationKandD}, we prove that we can write the $n-$th iteration of the mM-MCMC transition kernel with direct reconstruction as
	\begin{equation} \label{eq:mMkernelrewritten}
	\mathcal{K}_{mM}^n(x'|  x) = \nu(x'| \xi(x')) \ \mathcal{D}^n(\xi(x')| \xi(x)) + \mathcal{C}(\xi(x))^n \left(\delta(x'-x) - \nu(x'|\xi(x')) \ \delta\left(\xi(x')-\xi(x)\right)\right),
	\end{equation}
	with $C(\xi(x)) = 1 - \int_H \alpha_{CG} (z|\xi(x)) \ q_0(z | \xi(x)) dz$.
	This form for the $n$-th iteration of the micro-macro transition kernel has the advantage that it can be written as the product of the reconstruction distribution with the $n$-th iterate of the macroscopic invariant distribution plus another term. This form will come in handy later in the proof when we need to integrate the micro-macro transition kernel.

	Subtracting the invariant measure $\mu(x') = \nu(x'| \xi(x')) \ \mu_0(\xi(x'))$ from this expression and defining the constant $\eta$ is defined as $ \sup_{x \in \mathbb{R}^d} \norm{\delta(\cdot-x) + \nu(\cdot|\xi(\cdot)) \delta\left(\xi(\cdot)-\xi(x)\right)}_{TV}$, we can bound the total variation distance as
	\begin{equation} \label{eq:boundedTV}
	\begin{aligned}
	\norm{\mathcal{K}_{mM}(\cdot|  x) -\mu(\cdot)}_{TV} &\leq \norm{\nu(\cdot|\xi(\cdot)) \left(\mathcal{D}^n(\cdot|\xi(x))  + \mu_0(\cdot)\right)}_{TV}  \\ &+ \norm{\mathcal{C}(\xi(x))^n \left(\delta(x'-x) - \nu(x'|\xi(x')) \ \delta\left(\xi(x')-\xi(x)\right)\right) }_{TV} \\
	&\leq \norm{\nu(\cdot| \xi(\cdot))  \left(\mathcal{D}^n(\xi(\cdot)| \xi(x)) - \mu_0(\cdot)\right)}_{TV} + \mathcal{C}(\xi(x))^n \ \eta.
	\end{aligned}
	\end{equation}
	Note that the constant $\eta$ is finite because $\nu$ is bounded and the total variation distance of a delta function is $1$.
	
	We now bound each of the two terms in~\eqref{eq:boundedTV} independently. Using a property of the total variation distance~\cite[Prop. 3(a)]{roberts2004general}, we rewrite the first term as
	\begin{equation*}
	\begin{aligned}
	\norm{\nu(\cdot| \xi(\cdot))  \Big(\mathcal{D}^n(\xi(\cdot)| \xi(x)) - \mu_0(\cdot)\Big)}_{TV}  &= \sup_{g: \ \mathbb{R}^d\to [0,1]} \left| \int_{\mathbb{R}^d} g(y) \ \nu(y|\xi(y))  \Big(\mathcal{D}^{n}(\xi(y)|\xi(x)) - \mu_0(\xi(y))\Big) dy \right| \\
	&= \sup_{g: \ \mathbb{R}^d\to [0,1]} \left| \int_{H}  \left(\mathcal{D}^{n} (z|\xi(x)) - \mu_0(z)\right) \int_{\Sigma(z)} \frac{g(x) \nu(x|z)}{\norm{\nabla \xi(x)}} d\sigma_{\Sigma(z)}(x) dz \right| \\
	&= \sup_{g: \ \mathbb{R}^d\to [0,1]} \left| \int_{H} \mathbb{E}_{\nu}[g](z) \left(\mathcal{D}^{n} (z|\xi(x)) - \mu_0(z)\right) dz \right| \\
	&\leq \sup_{\tilde{g}: \ H \to [0,1]}  \left| \int_{H} \tilde{g}(z) \left(\mathcal{D}^{n} (z|\xi(x)) - \mu_0\right) dz \right| \\
	&= \norm{\mathcal{D}^{n}(\cdot |\xi(x)) - \mu_0(\cdot)}_{TV}.
	\end{aligned}
	\end{equation*}
	The first equality stems from the definition of total variation norm~\cite{roberts2004general}, and the second equality is due to the co-area formula~\eqref{eq:coarea}. Indeed, we have that the reconstruction distribution is given by $\nu(y|\xi(y)) = \mu(y)/\mu_0(\xi(y))$ inside the integral over $\mathbb{R}^d$~\eqref{eq:exactreconstruction2}, and this expression becomes $\mu(y)/\mu_0(z) \norm{\nabla \xi(y)}^{-1}$ inside the integral over $\Sigma(z)$. On the third line, we define the temporary variable
	\[
	\mathbb{E}_\nu[g](z) = \int_{\Sigma(z)} \frac{g(x) \nu(x|z)}{\norm{\nabla \xi(x)}} d\sigma_{\Sigma(z)}(x) \in [0,1],
	\]
	and the inequality on the fourth line is because we take the supremum over a possibly larger class of bounded functions between 0 and 1.
	
	To bound the the factor $\mathcal{C}(\xi(x))$ in the second term of~\eqref{eq:boundedTV}, we use compactness of $H$. The functions $q_0$ and $\alpha_{CG}$ are strictly positive and hence the integral
	\[
	\int_H \alpha_{CG}(y|\xi(x)) \ q_0(y| \xi(x)) dy
	\]
	is strictly greater than $0$ for every $\xi(x) \in H$. By compactness, the infimum of this integral for all $\xi(x)$ is hence also strictly positive, proving that the supremum of $\mathcal{C}(\xi(x))$ is strictly smaller than $1$, i.e., $\sup_{x \in \mathbb{R}^d} \mathcal{C}(\xi(x)) = \kappa < 1$. Putting both bounds together, we conclude that
	\[
	\norm{\mathcal{K}^n_{mM}(\cdot| \ x) -\mu(\cdot)}_{TV} \leq \norm{\mathcal{D}^n(\xi(\cdot)| \xi(x)) - \mu_0(\cdot)}_{TV} + \eta \ \kappa^n,
	\]
	proving the theorem.
\end{proof}
A consequence of Theorem~\ref{thm:ergodicity_mM-MCMC} is that mM-MCMC inherits all ergodicity properties from the macroscopic MCMC sampler. For instance, when $\mathcal{D}$ is uniformly ergodic, that is, we can bound $\norm{\mathcal{D}^n(\cdot|z) - \mu_0(z)}_{TV}$ by $C  \rho^n$ with $C > 0$ and $\rho < 1$, then the mM-MCMC method is also uniformly ergodic since we can bound $\norm{\mathcal{K}_{mM}(\cdot|  x) -\mu(\cdot)}_{TV} $ uniformly in $x$ by $(C + \eta) \ \max\{\rho, \kappa\}^n$. We can draw a similar conclusion for other types of ergodicity of the macroscopic sampler. We will numerically show in Section~\ref{subsubsec:effgaineps}, that the mM-MCMC method with exact reconstruction indeed results in more efficient results than when the exact reconstruction property is not satisfied.

\section{Numerical illustrations} \label{sec:results}
In this section, we numerically investigate the efficiency of the mM-MCMC scheme with direct reconstruction on two molecular problems: a three-atom molecule and the molecule butane. We compare the efficiency gain over the microscopic MALA (Metropolis-adjusted Langevin) method, where we specifically study the impact of three design choices in the mM-MCMC scheme: the macroscopic invariant distribution $\bar{\mu}_0$, the macroscopic proposal distribution $q_0$ and the reconstruction distribution $\bar{\nu}$. The efficiency gain criterion for a proper comparison of mM-MCMC with the MALA method is explained in Section~\ref{subsec:effcriterion}, and the numerical results for the three-atom molecule and butane are shown in Sections~\ref{subsec:threeatom} and~\ref{subsec:butane} respectively.

\subsection{Efficiency criterion} \label{subsec:effcriterion}
Consider a scalar function $F: \mathbb{R}^d \to \mathbb{R}$ and suppose we are interested in the average of $F$ with respect to the Gibbs measure $\mu$,
\[
\mathbb{E}_{\mu}[F] = \int_{\mathbb{R}^d} F(x) \ d\mu(x).
\]
If we sample the invariant measure $\mu$ using an MCMC method, we can estimate the above value as $\hat{F} = N^{-1} \sum_{n=1}^N F(x_n)$ with an ensemble of microscopic samples $\{x_n\}_{n=1}^N$. The variance on this estimate is
\begin{equation} \label{eq:varF}
\text{Var}[\hat{F}] = \frac{\sigma_F^2 \ K_{\text{corr}}}{N},
\end{equation}
where $\sigma_F^2$ is the inherent variance of $F$, 
\[
\sigma_F^2 = \int_{\mathbb{R}^d} \ \left(F(x) - \mathbb{E}_{\mu}[F] \right)^2 d\mu(x),
\]
and $K_{\text{corr}}$ is defined as
\[
K_{\text{corr}} = 1 + \frac{2}{\sigma_F^2} \ \sum_{n=1}^N \ \mathbb{E}\left[(F(x_n) - \mathbb{E}_{\mu}[F(x_n)])(F(x_0) - \mathbb{E}_{\mu}[F(x_0)] )\right],
\]
with $x_0$ the initial value of the Markov chain~\cite{meyn2012markov} .

The extra factor $K_{\text{corr}}$ in~\eqref{eq:varF} arises because the Markov chain Monte Carlo samples are not independent of each other. The higher $K_{\text{corr}}$, the more dependent the MCMC samples and the higher the variance~\eqref{eq:varF}. Another interpretation of the correlation parameter is that the `effective' number of samples is $N/ K_{\text{corr}}$.

In the following numerical experiments, we are interested in reducing the variance on the estimator $\hat{F}$ with mM-MCMC using the same runtime, compared the the microscopic MALA algorithm. Equivalently, we want to increase the effective number of samples $N/K_{\text{corr}}$ for a fixed runtime. We therefore define the efficiency gain of mM-MCMC over the microscopic MALA method as
\begin{equation} \label{eq:effgain}
\text{Gain} = \frac{\text{Var}[\hat{F}]_{\text{micro}}}{\text{Var}[\hat{F}]_{\text{mM}}} \ \frac{T_{\text{micro}}^N}{T_{\text{mM}}^N} = \frac{K_{\text{corr, micro}}}{K_{\text{corr, mM}}} \ \frac{T_{\text{micro}}^N}{T_{\text{mM}}^N}.
\end{equation}
Here, $T_{\text{micro}}^N$ is the measured execution time of the microscopic MCMC method for a fixed number of sampling steps $N$ and $T_{\text{mM}}^N$ is the measured execution time for the same number of steps $N$ of the mM-MCMC scheme. 

Usually, the execution time $T_{mM}^N$ of the mM-MCMC method with direct reconstruction is lower than the execution time $T_{micro}^N$ for the MALA method for the same number of sampling steps $N$, because the mM-MCMC method generates fewer distinct microscopic samples. Indeed, when a reaction coordinate value is rejected at the macroscopic level, we immediately keep the current microscopic sample without evaluating the (expensive) microscopic potential energy. The exact decrease in execution time for the same number of sampling steps depends on macroscopic acceptance rate. We will also show that the effective number of samples $N/K_{\text{corr, mM}}$ of the mM-MCMC scheme is usually orders of magnitude higher than the effective number of samples of the MALA scheme, $N/K_{\text{corr, micro}}$. By first sampling a reaction coordinate value with a large time steps $\Delta t$ at the macroscopic level, the correlation between two reaction coordinate values will, on average, be lower than the correlation between two microscopic samples generated with a small time step by the MALA method. Currently, however, we have no analytic formulas linking the correlations at the macroscopic and microscopic levels, so we will demonstrate this claim numerically. Combining both the decrease in execution time for the same number of sampling steps and the increase of the effective number of samples, we expect that mM-MCMC will be able to gain over the MALA method for moderate to large time-scale separations. We will show that the higher the time-scale separation, the larger the efficiency gain will be.

\subsection{The three-atom molecule} \label{subsec:threeatom}
\paragraph{\textbf{Model problem}}
In this section, we consider the mM-MCMC algorithm on a simple, academic, three-atom molecule, as first introduced in~\cite{legoll2010effective}. The three-atom molecule has a central atom $B$, that we fix at the origin of the two-dimensional plane, and two outer atoms, $A$ and $C$. To fix the superfluous degrees of freedom, we constrain atom $A$ to the $x-$axis, while $C$ can move freely in the plane. The three-atom molecule is depicted on Figure~\ref{fig:triatommolecule2}. 
\begin{figure}
	\centering
	\includegraphics[width=0.3\linewidth]{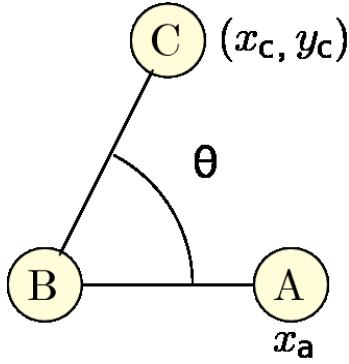}
	\caption{The three-atom molecule. Atom $A$ is constraint to the $x$-axis with $x$-coordinate $x_a$, atom $B$ is fixed at the origin of the plane and atom $C$ lies om the two-dimensional plane with Cartesian coordinates $(x_c, y_c)$.}
	\label{fig:triatommolecule2}
\end{figure}

The potential energy for the three-atom system consists of three terms,
\begin{equation} \label{eq:trheeatompotential}
V(x_a, x_c, y_c) = \frac{1}{2\varepsilon} \ (x_a-1)^2 + \frac{1}{2\varepsilon} \ (r_c-1)^2 + \frac{208}{2}\left(\left(\theta-\frac{\pi}{2}\right)^2 - 0.3838^2\right)^2,
\end{equation}
where $x_a$ is the $x-$coordinate of atom $A$ and $(x_c, y_c)$ are the Cartesian coordinates of atom $C$. The bond length $r_c$ between atoms $B$ and $C$ and the angle $\theta$ between atoms $A$, $B$ and $C$ are defined as
\begin{equation*}
\begin{aligned}
r_c &= \sqrt{x_c^2 + y_c^2} \\ 
\theta &= \arctantwo (y_c, x_c).
\end{aligned}
\end{equation*}
The first term in~\eqref{eq:trheeatompotential} describes the vibrational potential energy of the bond between atoms $A$ and $B$, with equilibrium length $1$. Similarly, the second term describes the vibrational energy of the bond between atoms $B$ and $C$ with bond length $r_c$. Finally, the third term determines the potential energy of the angle $\theta$ between the two outer atoms, which has an interesting bimodal behaviour. The distribution of $\theta$ has two peaks, one at $\frac{\pi}{2} -0.3838$ and another at $\frac{\pi}{2} + 0.3838$. 

The reaction coordinate that we consider in this section is the angle $\theta$, i.e.,
\begin{equation}
\xi(x) = \theta(x),
\end{equation}
since this variable is the slow component of the three-atom molecule. Additionally, the angle $\theta$ is also independent of the time-scale separation, given by $\varepsilon$.

\paragraph{\textbf{Outline of this section}}
In the next set of experiments, we study the effect of choice of the approximate macroscopic invariant distribution $\bar{\mu}_0$, the macroscopic transition distribution $q_0$ and the reconstruction distribution $\bar{\nu}$ on the efficiency of the resulting mM-MCMC scheme. More specifically, we investigate the efficiency gain of mM-MCMC on the estimated expected value of $\theta$, for different parameter settings and two values of the small-scale parameter $\varepsilon$: $\varepsilon=10^{-4}$ and $\varepsilon=10^{-6}$.

 In Section~\ref{sec:direct_reconstruction_illustration}, we will define the exact free energy of $\theta$ and the exact reconstruction distribution $\nu$ such that exact reconstruction~\eqref{eq:exactreconstruction2} holds. We also visually illustrate the superior performance of mM-MCMC over the microscopic MALA method on a histogram fit of the bimodal distribution of $\theta$. In Section~\ref{subsubsec:Aq}, we subsequently consider the impact of three choices of the approximate macroscopic distribution $\bar{\mu}_0$, combined with two choices of the macroscopic transition distributions $q_0$ on the efficiency gain of mM-MCMC, while we keep the correct reconstruction distribution $\nu$~\eqref{eq:conditionalxgivenz} fixed. Third, in Section~\ref{subsubsec:Anu}, we fix the macroscopic transition distribution $q_0$ and study the impact of the two choices for the reconstruction distribution $\bar{\nu}$, combined with the same three options for the approximate macroscopic distribution $\bar{\mu}_0$ used in Section~\ref{subsubsec:Aq} on the efficiency gain of mM-MCMC. Finally, we investigate the efficiency gain of mM-MCMC over the MALA algorithm for a range of time-scale separations in Section~\ref{subsubsec:effgaineps}.

\subsubsection{Visual inspection of the performance of mM-MCMC\label{sec:direct_reconstruction_illustration}}
In the three-atom molecule, the exact free energy function of $\theta$ is directly visible in the potential energy function~\eqref{eq:trheeatompotential}. Indeed, the exact time-invariant distribution of the chosen reaction coordinate is
\begin{equation} \label{eq:threeatommacroinvariant}
\mu_0(\theta) \propto \exp\left(-\beta A(\theta)\right), \ \ A(\theta) = \frac{208}{2}\left(\left(\theta-\frac{\pi}{2}\right)^2 - 0.3838^2\right)^2.
\end{equation}
Therefore, we use the exact invariant distribution of the reaction coordinates in the mM-MCMC method, i.e., $\bar{\mu}_0 = \mu_0$.
When the angle $\theta$ of a molecule is given, for example after the macroscopic proposal step in the mM-MCMC method, we need to reconstruct the position $x_a$ of molecule $A$ and the bond length $r_c$ between atoms $B$ and $C$. To maximize the efficiency gain of mM-MCMC, we choose the reconstruction distribution such that the microscopic acceptance rate is always $1$. We thus choose the reconstruction distribution such that the exact reconstruction property holds~\eqref{eq:exactreconstruction2}, i.e., we take $\bar{\nu} = \nu$, as defined in~\eqref{eq:conditionalxgivenz}. For the three-atom molecule, this reconstruction distribution takes the form
\begin{equation} \label{eq:threeatomreconstruction}
\nu(x_a, r_c | \ \theta) \propto \exp\left(-\frac{\beta}{2\varepsilon} \ (x_a-1)^2\right) \exp\left(-\frac{\beta}{2\varepsilon} \ (r_c-1)^2\right),
\end{equation}
which is defined on the sub-manifold $\Sigma(\theta)= \{(x_a, x_c, y_c) \in \mathbb{R}^3 | \ x_c = r \cos(\theta), y_c=r \sin(\theta), r \geq 0\}$ of constant reaction coordinate value. While sampling the reconstruction distribution, we also need to make sure that $r_c$ is positive. Therefore, we first sample $r_c$ from the Gaussian distribution with mean $1$ and variance $\varepsilon \beta^{-1}$ and reject the proposal whenever $r_c$ is negative. If the proposed value is positive, we accept.

Furthermore, we base the macroscopic proposals on the exact effective dynamics of $\theta$ with time step $\Delta t = 0.01$. In this model problem, the effective dynamics of $\theta$ reduces to an overdamped Langevin equation
\begin{equation} \label{eq:effdyntheta}
d\theta = -\nabla A(\theta) dt + \sqrt{2 \beta^{-1}} dW,
\end{equation}
 since one can easily verify that $\norm{\nabla \theta(x)} = 1$ for all $x \in \mathbb{R}^3$. The initial condition to the Markov chain is $(x_a, x_c, y_c) = (1, 0, 1)$.

For the numerical experiment, we choose two values for the time-scale separation parameter, $\varepsilon = 10^{-4}$ and $\varepsilon = 10^{-6}$ and sample the three-atom molecule with $N=10^6$ microscopic samples. For simplicity, we define the inverse temperature parameter as $\beta=1$. The proposals of the microscopic MALA method are based on the Euler-Maruyama discretization of the overdamped Langevin dynamics~\eqref{eq:overdamped_langevin} with microscopic time step size $\delta t = \varepsilon$. On Figure~\ref{fig:threeatom_mM_direct}, we visually compare the histograms of $\theta$ obtained by MALA and mM-MCMC.

\begin{figure}
	\centering
	\begin{subfigure}[b]{0.5\textwidth}
		\centering
		\includegraphics[width=0.85\linewidth]{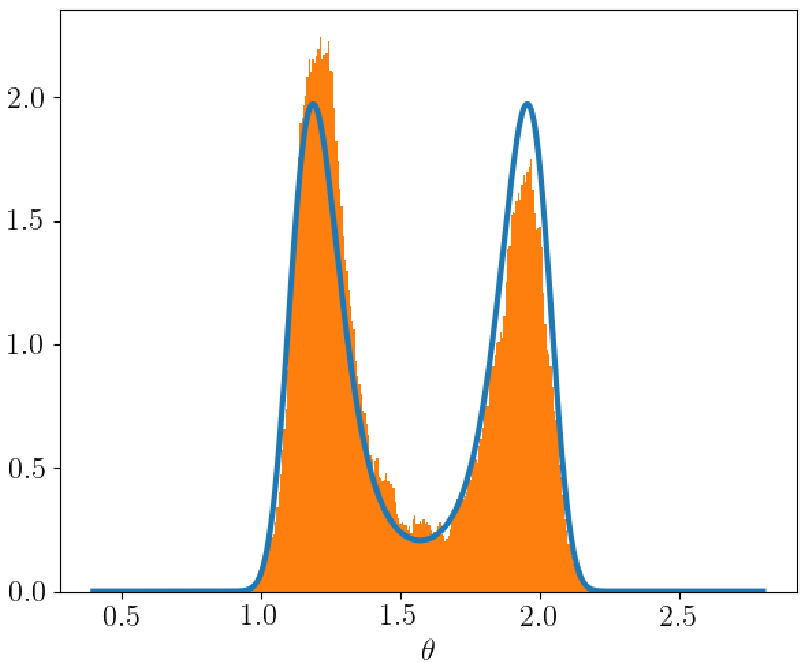}
	\end{subfigure}%
	\begin{subfigure}[b]{0.5\textwidth}
		\centering
		\includegraphics[width=0.85\linewidth]{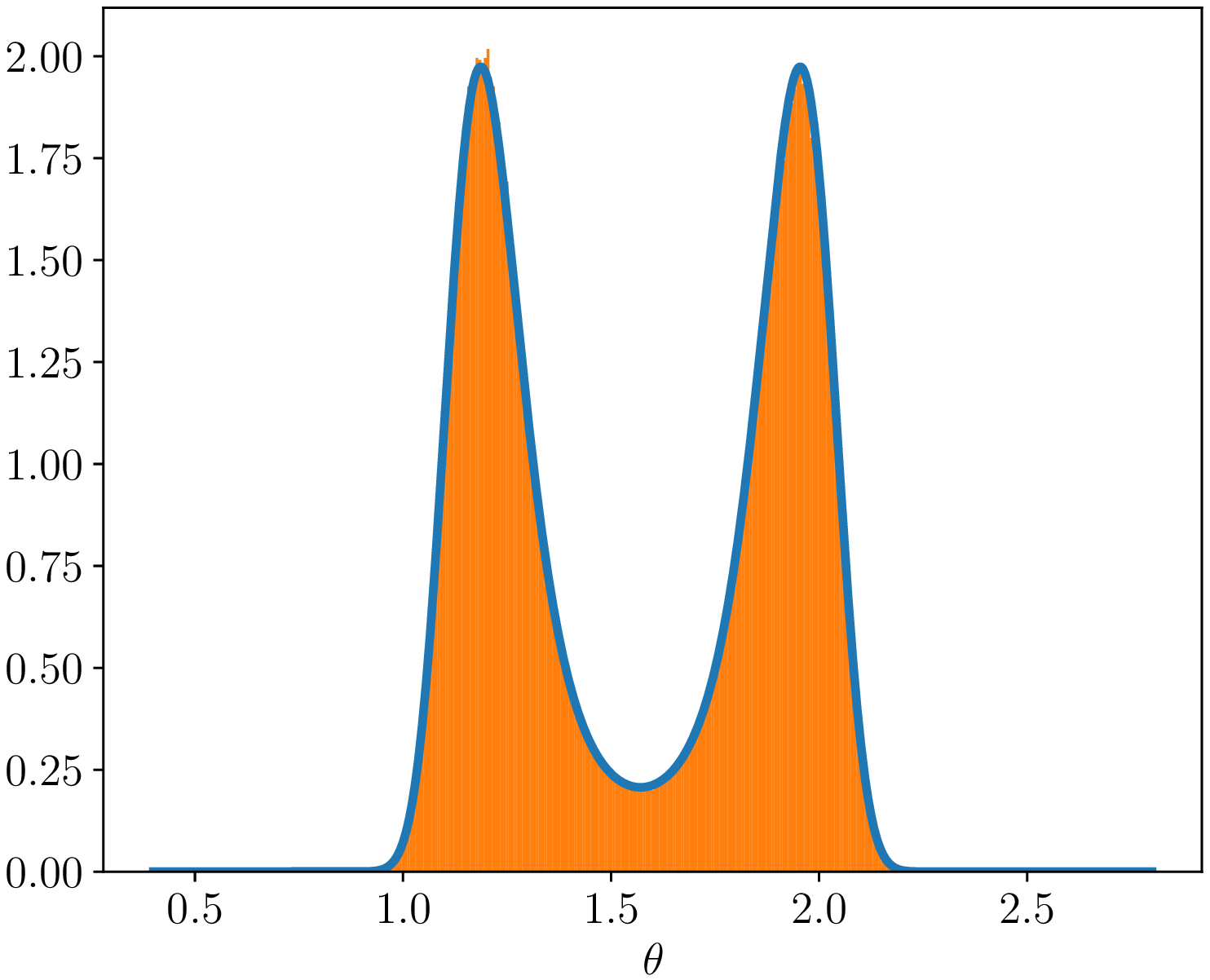}
	\end{subfigure}
	\begin{subfigure}[b]{0.5\textwidth}
		\centering
		\includegraphics[width=0.85\linewidth]{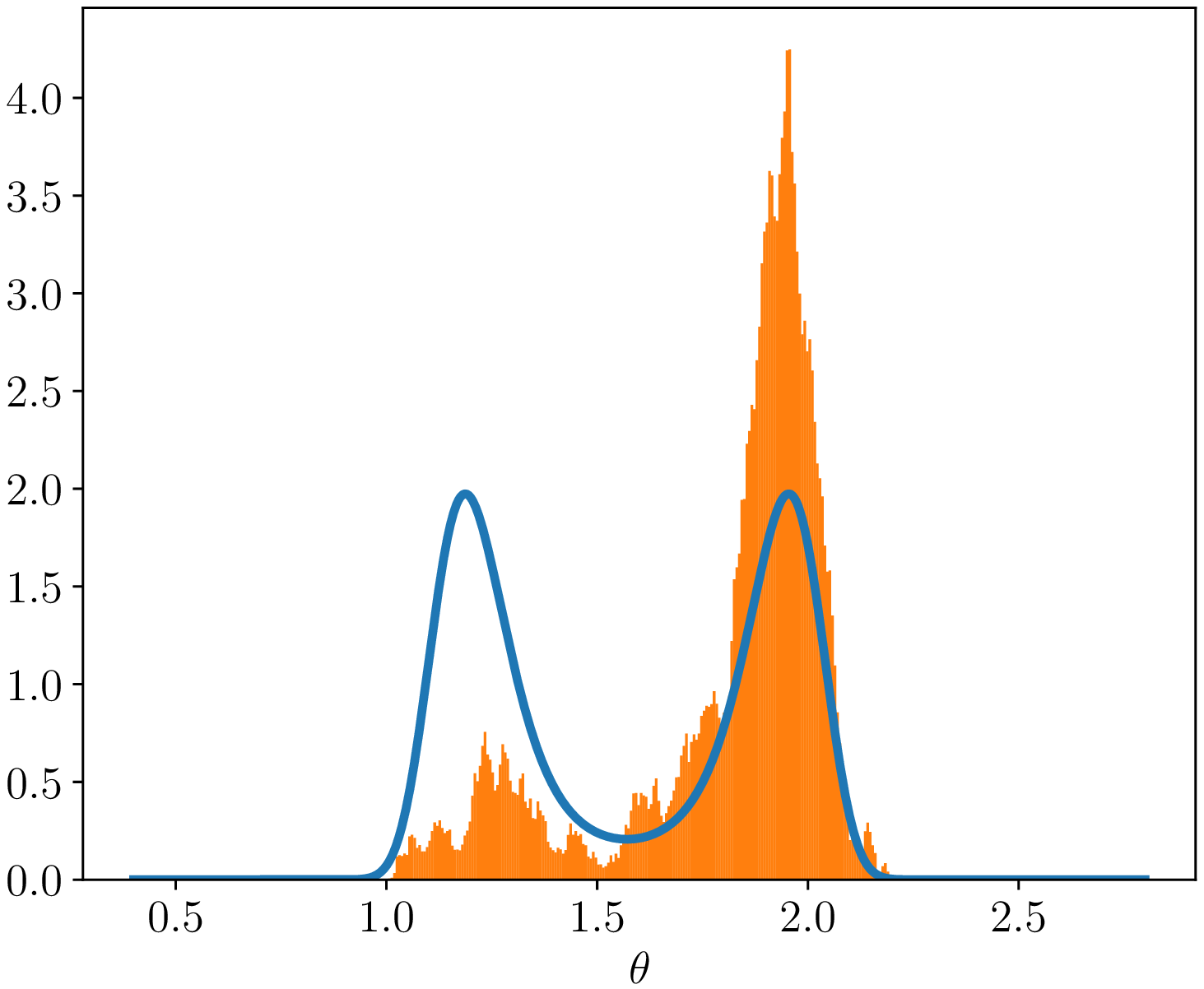}
	\end{subfigure}%
	\begin{subfigure}[b]{0.5\textwidth}
		\centering
		\includegraphics[width=0.85\linewidth]{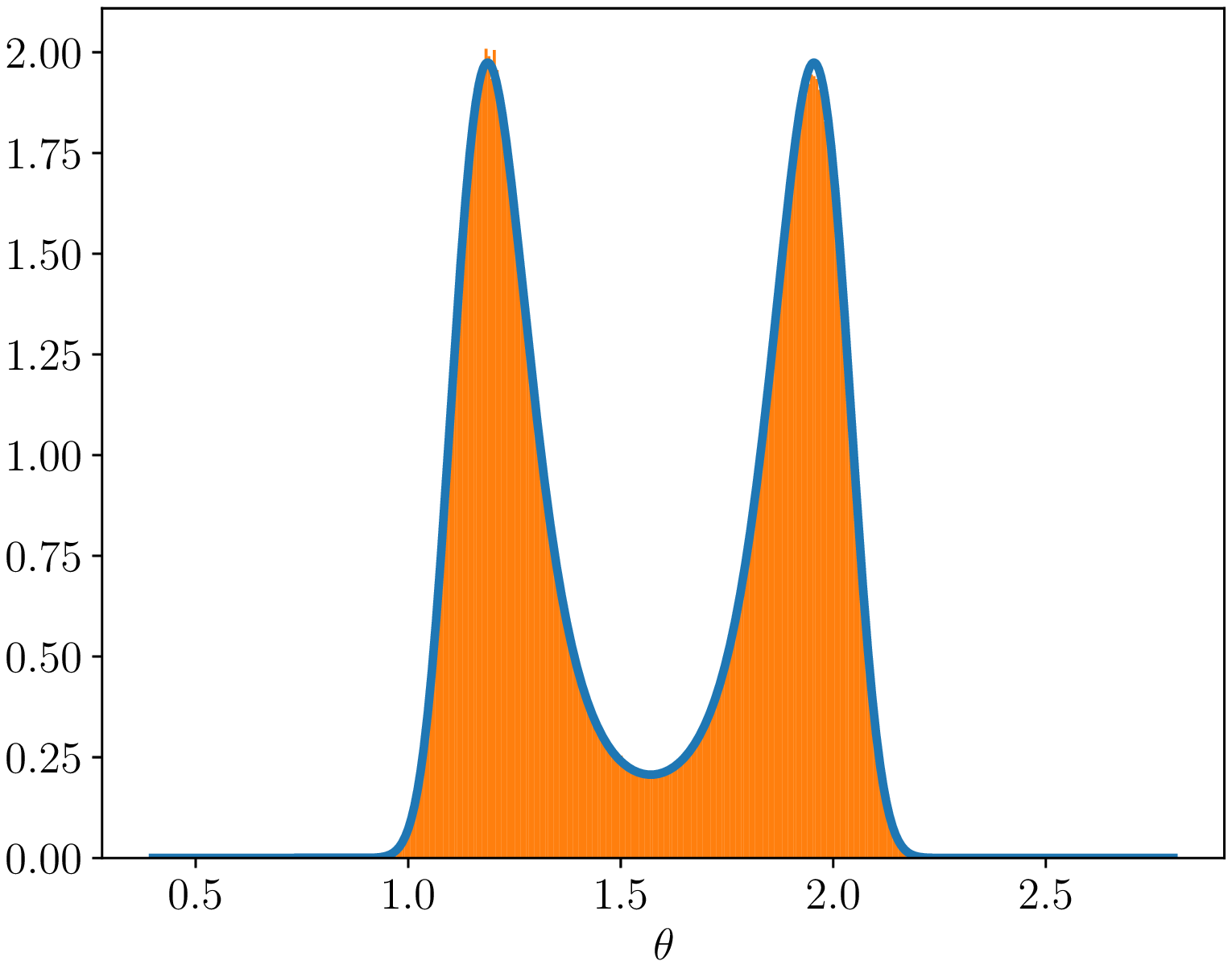}
	\end{subfigure}
	\caption{Visual representation of the histogram of $\theta$ of the microscopic MCMC method (MALA) (left) and mM-MCMC with direct reconstruction (right) on the three-atom molecule, with reaction coordinate $\theta$. The simulation parameters are $\varepsilon=10^{-5}$ (top) and $\varepsilon=10^{-6}$ (bottom) and the number of samples is $N=10^6$. The MALA method remains stuck in the potential well of $\theta$ around $\pi/2+0.3838$, while the mM-MCMC method samples the distribution well.}
	\label{fig:threeatom_mM_direct}
\end{figure}

Clearly, the MALA method remains stuck for a long time in the potential well around $\pi/2+0.3838$ when $\varepsilon=10^{-6}$. Expectedly, the mixing improves, however, when $\varepsilon$ increases to $10^{-4}$. The mM-MCMC method is able to sample the distribution of $\theta$ accurately, regardless of the time-scale separation.

\subsubsection{Impact of $\bar{A}$ and $q_0$ on the efficiency gain} \label{subsubsec:Aq}
\paragraph{\textbf{Experimental setup}}
In the second numerical experiment on the three-atom molecule, we specifically investigate the impact of the choice of approximate macroscopic distribution $\bar{\mu}_0$ (in the form of an approximate free energy) and the choice of macroscopic transition distribution $q_0$ on the efficiency gain of mM-MCMC. We define three choices for the approximate macroscopic distribution and for two choices of the macroscopic transition distribution. The three choices for the approximate free energy functions are
%\begin{equation}
\begin{align} \label{eq:freeenergies}
\bar{A}^1(\theta) &= \frac{208}{2}\left(\left(\theta-\frac{\pi}{2}\right)^2 - 0.3838^2\right)^2, \nonumber \\ %%%%%%%%
\bar{A}^2(\theta) &= \frac{208}{2}\left(\left(\theta-\frac{\pi}{2}\right)^2 - 0.4838^2\right)^2, \\
\bar{A}^3(\theta) &= \frac{208}{2}\left(\left(\theta-\frac{\pi}{2}\right)^2 - 0.3838^2\right)^2 + \cos(\theta). \nonumber
\end{align}
%\end{equation}
The first formula is the exact free energy function as stated in Section~\ref{sec:direct_reconstruction_illustration}. The second expression for the approximate free energy is obtained by perturbing the two peaks of the exact free energy~\eqref{eq:freeenergy} by a distance $0.1$ of radians. Finally, $\bar{A}^3$ is obtained from the exact free energy expression by adding the cosine function to the exact free energy $\bar{A}^1$, resulting in a large perturbation on the amplitude of the associated macroscopic invariant distribution. The effect of the cosine perturbation is that the hight of the left peak in the macroscopic distribution of $\theta$ is decreased, while the height of the right peak is increased. The three Gibbs distributions associated to these options for the (approximate) free energy with $\beta=1$ are shown in Figure~\ref{fig:freeenergy}.

\begin{figure}
	\centering
	\includegraphics[width=0.6\linewidth]{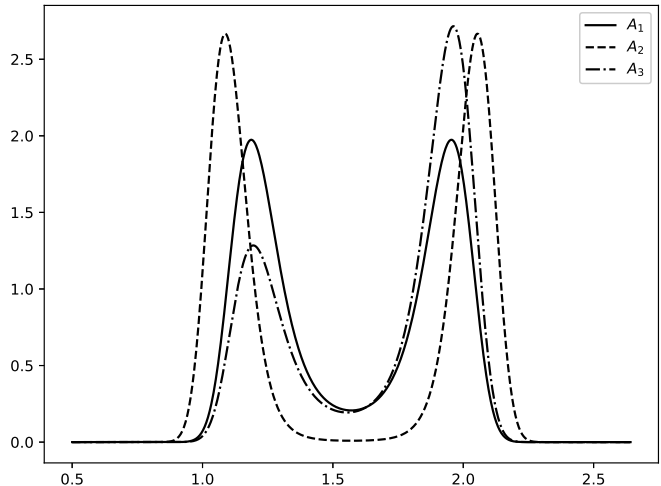}
	\caption{The three approximate free energy distributions as defined in equation~\eqref{eq:freeenergies}. The blue curve is the exact free energy of the three-atom molecule~\eqref{eq:trheeatompotential}, the red curve is obtained by shifting the local maxima of the exact distribution by $0.1$ by the left and the right, while the final distribution is the exact expression perturbed by the cosine function.}
	\label{fig:freeenergy}
\end{figure}

Similarly, the two choices for the macroscopic proposal distribution $q_0$ are based on the following two stochastic dynamical systems for the reaction coordinate,
%\begin{equation*}
\begin{align} \label{eq:coarseproposalmoves}
q_0^1  &:d\theta = -\nabla A_i(\theta)dt + \sqrt{2\beta^{-1}} dW \nonumber \\ %%%%%%%%
q_0^2 &: d\theta = \sqrt{2\beta^{-1}} dW.
\end{align}
%\end{equation*}
The first stochastic differential equation is the overdamped Langevin dynamics for each of the approximate free energy functions $\bar{A}^i, i=1,2,3$ (MALA), while the latter equation is a simple Brownian motion in the reaction coordinate space.

In the following experiment, we run the mM-MCMC algorithm with each of these six combinations for the approximate macroscopic invariant distribution $\bar{\mu}_0 \propto \exp\left(-\beta\bar{A}\right)$ and the macroscopic transition distribution $q_0$ for $N=10^6$ sampling steps and with a macroscopic time step $\Delta t = 0.01$. For each of these combinations, we compute the total macroscopic acceptance rate, the microscopic acceptance rate after reconstruction, the average runtime, the numerical variance on the estimated mean of $\theta$ and the total efficiency gain of of mM-MCMC over the MALA algorithm, as explained in Section~\ref{subsec:effcriterion}. The microscopic time step for the MALA method is $\delta t = \varepsilon$ and for we choose $\beta=1$ for the inverse temperature. For a statistically good comparison, we average the results over $100$ independent runs. The numerical efficiency gains are depicted in Tables~\ref{tab:threeatomgain1e-04} and~\ref{tab:threeatomgain1e-06} for $\varepsilon=10^{-4}$ and $\varepsilon=10^{-6}$ respectively.

\begin{table}[h]
\centering
\begin{tabular}{c|c|c|c|c|c}
\centering
Parameters & \pbox{15cm}{Macroscopic \\ acceptance rate} & \pbox{15cm}{Microscopic \\ acceptance rate} & \pbox{15cm}{Runtime \\ gain} & \pbox{15cm}{Variance \\ gain} & Total efficiency gain \\
\hline
Langevin, $\bar{A}^1$ & 0.749932            &           1            &      2.45692     &      85.3266       &             209.64 \\
Langevin, $\bar{A}^2$ &                      0.730384              &         0.432508      &     2.62306  &        28.4122          &           74.527 \\
Langevin, $\bar{A}^3$    &                   0.749653         &              0.950238     &      1.95663       &    99.7186            &       195.112 \\
Brownian, $\bar{A}^1$    &                   0.645188           &            1                 & 3.03108        &      85.1586               &          258.122 \\
Brownian, $\bar{A}^2$     &                 0.61375             &           0.597058       &    3.28045 &         35.6794         &             117.044 \\
Brownian, $\bar{A}^3$        &               0.645654            &           0.959794       &    2.72728     &       81.3405        &           221.838 \\
\end{tabular}
\caption{A summary of different statistics of the mM-MCMC method with $\varepsilon=10^{-4}$ for six combinations of the (approximate) macroscopic invariant distribution and macroscopic proposal moves. The `Macroscopic acceptance rate' column is the number of accepted reaction coordinate values relative to the number of microscopic samples. Second, the `Microscopic acceptance rate' counts the number of microscopic samples that are accepted after reconstruction, relative to all macroscopic accepted reaction coordinate values. The `Runtime gain' column is given by the average mM-MCMC runtime divided by the average Metropolis-Hastings runtime. while the `Variance gain' column summarizes the gain in variance of the computed means of $\theta$ with mM-MCMC over the variance of the computed means of $\theta$ with the MALA method. The final column combines the previous columns using the efficiency gain criterion~\eqref{eq:effgain}.}
\label{tab:threeatomgain1e-04}
\end{table}

\begin{table}[h]
\centering
\begin{tabular}{c|c|c|c|c|c}
	\centering
	Parameters & \pbox{15cm}{Macroscopic \\ acceptance rate} & \pbox{15cm}{Microscopic \\ acceptance rate} & \pbox{15cm}{Runtime \\ gain} & \pbox{15cm}{Variance \\ gain} & Total efficiency gain \\
	\hline
	Langevin, $\bar{A}^1$         &              0.749906        &               1           &        2.50343       &     3297.65           &       8255.44 \\
	Langevin, $\bar{A}^2$           &             0.730538              &         0.43242    &        2.64621    &    933.64    &      2470.61 \\
Langevin, $\bar{A}^3$           &           0.749596               &        0.950308    &       1.99301     &  2461.08   &           4904.96 \\
Brownian, $\bar{A}^1$            &          0.645272    &                   1                &  3.058        & 3223.1             &     9856.26 \\
Brownian, $\bar{A}^2$                &        0.613716      &                 0.59717       &     3.31087    &    1274.01        &         4218.07  \\
Brownian, $\bar{A}^3$              &      0.645702        &               0.959848       &    2.77357  &   3229.06        &        8956.03 \\
\end{tabular}
\caption{A summary of different statistics of the mM-MCMC method with $\varepsilon=10^{-6}$ for six combinations of the (approximate) macroscopic invariant distribution and macroscopic proposal moves. The different columns are the same as in Table~\ref{tab:threeatomgain1e-04}.}
\label{tab:threeatomgain1e-06}
\end{table}

\paragraph{\textbf{Numerical results}}
The numerical results in Tables~\ref{tab:threeatomgain1e-04} and~\ref{tab:threeatomgain1e-06} indicate a large efficiency gain of mM-MCMC over the microscopic MALA algorithm. The efficiency gain is on the order of the time-scale separation $208/\varepsilon$. For instance, a gain of a factor $8255$ in Table~\ref{tab:threeatomgain1e-06} indicates that the mM-MCMC method needs $8255$ times fewer sampling steps to obtain the same variance on the estimated mean of $\theta$ than the microscopic MALA method, for the same runtime. 

First, note that for both values of $\varepsilon$, the macroscopic acceptance rate is lower when using Brownian macroscopic proposals than when using Langevin dynamics proposals. This result is intuitive since the Brownian motion does not take into account the underlying macroscopic probability distribution, while the Langevin dynamics will automatically choose reaction coordinate values in regions of higher macroscopic probability. However, with a lower macroscopic acceptance rate comes a lower runtime as well since we need to reconstruct fewer microscopic samples and hence fewer evaluations of the microscopic potential energy. This effect is indeed visible in the fourth column where the runtime gain is higher for Brownian motion than that of the corresponding Langevin dynamics. Further, one can see in both Tables that there is almost no difference in the gain on the variance of the estimated mean of $\theta$ between the two macroscopic proposal moves. Hence, the total efficiency gain of mM-MCMC is almost completely determined by the lower runtime due to the macroscopic Brownian proposals. Practically, however, we conclude that there is a small difference in efficiency gain between the macroscopic proposals based on Brownian motion or on the effective dynamics of $\theta$, as is visible in the last column of Tables~\ref{tab:threeatomgain1e-04} and~\ref{tab:threeatomgain1e-06}.

The choice of macroscopic invariant distribution $\bar{\mu}_0$, however, has a larger impact on the efficiency of mM-MCMC with direct reconstruction. First of all, one can see that the macroscopic acceptance rate is less affected by the choice of macroscopic invariant distribution than by the choice of macroscopic proposal move $q_0$. Indeed, the Langevin proposals are based on the free energy of their respective approximate macroscopic distribution. However, the microscopic acceptance rate is significantly affected by the approximate macroscopic distribution $\bar{\mu}_0$. Since the microscopic acceptance criterion is used to correct the microscopic samples from having the wrong macroscopic distribution, the more the approximate macroscopic distribution $\bar{\mu}_0$ deviates from the exact macroscopic distribution $\mu_0$, the lower the microscopic acceptance rate will be. Indeed, the microscopic acceptance rate for the approximate free energy $\bar{A}^2$ is much lower than that for $\bar{A}^1$ since the local minima of $\bar{A}^2$ are located at different positions. On the other hand, the approximate free energy $\bar{A}^3$ lies closer to $\bar{A}^1$ since only the height of both peaks different, resulting in a microscopic acceptance rate close to $1$.

Consequently, the closer $\bar{\mu}_0$ lies to the exact invariant distribution $\mu_0$ of the reaction coordinates, the higher the gain in variance is over the microscopic MALA method. Indeed, if the microscopic acceptance rate is low, we store the same microscopic sample many times, prohibiting a thorough exploration of the microscopic state space and thus keeping the variance obtained by mM-MCMC high.

 We thus conclude this experiment by stating that the choice of approximate macroscopic distribution $\bar{\mu}_0$ has a larger impact on the efficiency gain of mM-MCMC than the choice of macroscopic transition distribution $q_0$.

\subsubsection{Impact of $\bar{A}$ and $\bar{\nu}$ on the efficiency gain} \label{subsubsec:Anu}
\paragraph{\textbf{Experimental setup}}
For the third experiment on the three-atom molecule, we investigate the effect of the choice of reconstruction distribution $\bar{\nu}$ on the efficiency gain of mM-MCMC over MALA, in conjunction with the same three choices for the approximate free energy $\bar{A}$~\eqref{eq:freeenergies}. For consistency of the numerical results, we employ macroscopic proposal moves based on the effective dynamics of $\theta$~\eqref{eq:effdyntheta} with the given approximate free energy function. The two reconstruction distributions that we consider in this numerical experiment are
\begin{align} \label{eq:reconstructions}
\bar{\nu}^1(x|\theta) &= \nu(x|\theta) \nonumber \nonumber \\
\bar{\nu}^2(x|\theta) &\propto \exp\left(-\frac{(x_a-1)^2}{4\varepsilon}\right) \exp\left(-\frac{(r_c-1)^2}{4\varepsilon}\right).
\end{align}
The first reconstruction distribution is the exact time-invariant distribution as defined in~\eqref{eq:exactreconstruction2}, while the second distribution is obtained by increasing the variance on $x_a$ and $r_c$ by a factor of $2$, relative to $\bar{\nu}^1$.

\begin{table}[h]
	\centering
	\begin{tabular}{c|c|c|c|c|c}
		\centering
		Parameters & \pbox{15cm}{Macroscopic \\ acceptance rate} & \pbox{15cm}{Microscopic \\ acceptance rate} & \pbox{15cm}{Runtime \\ gain} & \pbox{15cm}{Variance \\ gain} & Total efficiency gain \\
		\hline
		$\bar{A}^1, \bar{\nu}^1$ &                            0.749932        &               1        &           2.45692         &    85.3266         &        209.64 \\
		$\bar{A}^2, \bar{\nu}^1$                &            0.730384     &                  0.432508   &         2.62306           &  28.4122             &     74.527 \\
		$\bar{A}^3, \bar{\nu}^1$        &                    0.749653    &                   0.950238      &      1.95663            & 99.7186        &         195.112 \\
		$\bar{A}^1, \bar{\nu}^2$ &                           0.749891       &                0.476963       &     2.45479            & 37.5736         &         92.2354 \\
		$\bar{A}^2, \bar{\nu}^2$ &                           0.730482                   &    0.266464        &    2.62922           &  17.2829              &    45.4405 \\
		$\bar{A}^3, \bar{\nu}^2$     &                       0.749654               &        0.474436       &     1.92062            & 50.257          &         96.5246 \\
	\end{tabular}
	\caption{A summary of different statistics of the mM-MCMC method with $\varepsilon=10^{-4}$ for six combinations of the (approximate) macroscopic invariant distribution and reconstruction distribution. For each of these six combinations, we record the average acceptance rate at the macroscopic level, the average acceptance rate on the microscopic level, conditioned on all accepted macroscopic samples and the gain in runtime and variance on the estimated mean of $\theta$ of mM-MCMC over the microscopic MALA method. The final column record the total efficiency gain of mM-MCMC, which is the product of the two former columns.}
	\label{tab:threeatomgainnu}
\end{table}

\begin{table}[h]
	\centering
	\begin{tabular}{c|c|c|c|c|c}
		\centering
		Parameters & \pbox{15cm}{Macroscopic \\ acceptance rate} & \pbox{15cm}{Microscopic \\ acceptance rate} & \pbox{15cm}{Runtime \\ gain} & \pbox{15cm}{Variance \\ gain} & Total efficiency gain \\
		\hline
		$\bar{A}^1, \bar{\nu}^1$            &                0.749906          &             1             &      2.50343           & 3297.65          &         8255.44 \\
		$\bar{A}^2, \bar{\nu}^1$                 &           0.730538             &          0.43242    &         2.64621            & 933.64         &          2470.61\\
		$\bar{A}^3, \bar{\nu}^1$  &                          0.749596          &             0.950308        &    1.99301           & 2461.08             &      4904.96 \\
		$\bar{A}^1, \bar{\nu}^2$              &              0.750029            &           0.47711    &         2.58486       &     1250.1  &                  3231.34 \\
		$\bar{A}^2, \bar{\nu}^2$        &                    0.730455         &              0.266443       &     2.65114         &    547.826           &       1452.37 \\
		$\bar{A}^3, \bar{\nu}^2$  &                          0.749564   &                    0.474496    &        1.93858     &       1396.05     &              2706.35 \\
	\end{tabular}
	\caption{A summary of different statistics of the mM-MCMC method with $\varepsilon=10^{-6}$ for six combinations of the (approximate) macroscopic invariant distribution and reconstruction distribution. The columns are the same as in Table~\ref{tab:threeatomgainnu}.}
	\label{tab:threeatomgainnu6}
\end{table}

In this experiment, we again compute the macroscopic acceptance rate, the microscopic acceptance rate, the gain in runtime for a fixed number of sampling steps, the gain in variance of the estimated mean of $\theta$ and the total efficiency gain of mM-MCMC over the microscopic MALA algorithm for each of the six combinations of the approximate macroscopic distribution $\bar{\mu}_0$ and reconstruction distribution $\bar{\nu}$. We perform each experiment with $N=10^6$ steps, the temperature parameter is $\beta=1$, the macroscopic time step is again $\Delta t=0.01$. We use a time step $\delta t = \varepsilon$ for the MALA algorithm. For a good comparison, the numerical results are averaged over $100$ independent runs. The numerical results are shown in Table~\ref{tab:threeatomgainnu} for $\varepsilon=10^{-4}$ and in Table~\ref{tab:threeatomgainnu6} for $\varepsilon=10^{-6}$.

\paragraph{\textbf{Numerical results}}
As intuitively expected, the choice of reconstruction distribution has a negligible impact on the macroscopic acceptance rate, but it does have a significant effect on the microscopic acceptance rate. For both values of $\varepsilon$ and for the three choices of $\bar{A}$, reconstruction distribution $\bar{\nu}^2$ has a lower microscopic acceptance rate than $\bar{\nu}^1$. Since the distribution $\bar{\nu}^2$ is not the correct reconstruction distribution~\eqref{eq:conditionalxgivenz}, the microscopic acceptance criterion needs to correct for this wrong reconstruction, lowering the average microscopic acceptance rate. As we also noted in the previous experiment, a lower microscopic acceptance rate keeps the variance on the estimated mean of $\theta$ high and hence the total efficiency gain of mM-MCMC over the microscopic MALA algorithm is small.

To conclude, the choice of approximate macroscopic invariant distribution $\bar{\mu}_0$ and the choice of reconstruction distribution $\bar{\nu}$ have a significant impact on the total efficiency gain of mM-MCMC. The closer $\bar{\mu}_0$ and $\bar{\nu}$ lie to $\mu_0$ and $\nu$, respectively, the higher the total efficiency gain will be. In practice, however, it can be hard to find a good expression for the approximate free energy and the reconstruction distribution. Even if we have a good approximate reconstruction distribution available, it can be computationally expensive and cumbersome to sample from this reconstruction distribution since the sub-manifold of constant reaction coordinate value $z$, $\Sigma(z)$, may have a highly non-linear form. We will therefore explore a general \emph{indirect} reconstruction step that samples, in a general manner, from a reconstruction distribution that lies close the exact reconstruction distribution $\nu$ in a companion paper~\cite{vandecasteele2020indirect}.

\subsubsection{Efficiency gain as a function of $\varepsilon$} \label{subsubsec:effgaineps}
\paragraph{\textbf{Experimental setup}}
In the fourth and final experiment on the three-atom molecule, we estimate the efficiency gain of mM-MCMC on the estimated mean and variance of the angle $\theta$, as a function of the time-scale separation $\varepsilon$. We consider two different choices for the approximate macroscopic distribution $\bar{\mu}_0$ and the reconstruction distribution $\bar{\nu}$. For the first choice, we take the exact free energy $\bar{A}^1$~\eqref{eq:freeenergies} and the exact time-invariant reconstruction distribution $\bar{\nu}^1$~\eqref{eq:reconstructions} such that the exact reconstruction property holds. The other choice consists of approximate free energy $\bar{A}^2$ and reconstruction distribution $\bar{\nu}^2$. For both choices, the macroscopic proposal moves are based on the overdamped Langevin dynamics with the corresponding (approximate) free energy function and time step $\Delta t = 0.01$. We measure the efficiency gain for four values of the time-scale separation, $\varepsilon = 10^{-i}, \ i=3,\dots,6$ and with $N=10^6$ sampling steps. The microscopic time step for the microscopic MALA algorithm is $\delta t = \varepsilon$, and the inverse temperature is $\beta=1.$. The numerical results are shown in Figure~\ref{fig:effgaindirect}.

\begin{figure}
	\centering
	\includegraphics[width=0.7\linewidth]{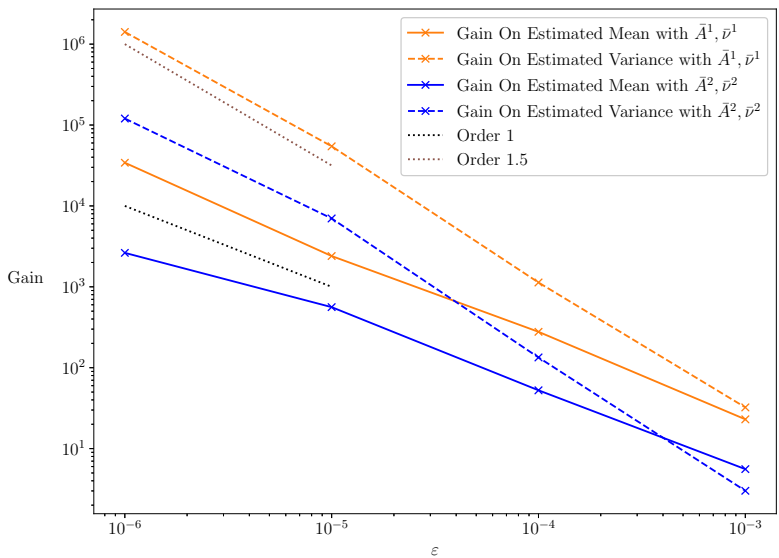}
	\caption{Efficiency gain of mM-MCMC over the standard MCMC method on the estimated mean (solid line) and variance (dashed line) of the angle $\theta$ for two different parameter choices: $(\bar{A}^1, \bar{\nu}^1)$ (orange lines) and $(\bar{A}^2, \bar{\nu}^2)$ (blue lines). The gain is computed using the criterion~\eqref{eq:effgain}, for $\varepsilon=10^{-i},  \ i=3,\dots,6$ and with $N=10^6$ microscopic samples. For both parameter choices, the efficiency gain on the estimated mean increases linearly with decreasing $\varepsilon$, and the gain on the estimated variance increases faster, approximately $\varepsilon^{-1.5}$. However, the efficiency gain on the estimated mean and variance of $\theta$ increases slower when the exact reconstruction property does not hold, i.e., $\bar{A}^2, \bar{\nu}^2$, than when it does, i.e., $\bar{A}^1, \bar{\nu}^1$.}
	\label{fig:effgaindirect}
\end{figure}

\paragraph{\textbf{Numerical results}}
For both choices of the parameters in the mM-MCMC method, the efficiency gain on the estimated mean and variance increases linearly or faster with decreasing $\varepsilon$, proving the mM-MCMC can accelerate the sampling of systems with a medium to large time-scale separation. Additionally, in case the exact reconstruction property~\eqref{eq:exactreconstruction5} holds with $(\bar{A}^1, \bar{\nu}^1)$, the efficiency gain is higher than when the exact reconstruction property does not hold, i.e., in case of $(\bar{A}^2, \bar{\nu}^2)$. This numerical result is an illustration of Theorem~\ref{thm:ergodicity_mM-MCMC}, which states that the rate of convergence of mM-MCMC is identical to the rate of convergence of the macroscopic sampler when the exact reconstruction property holds. When the latter assumption does not hold, the rate of convergence can be lower than the macroscopic rate. Currently, we have no way of deriving how the efficiency gain depends on the time-scale separation and we can only observe the numerical results in this manuscript. A good point to start such an analysis would be study how the variance on the estimated mean of $\theta$ of the microscopic MALA method depends on the time-scale separation. We defer such an analysis to further research.

\subsection{Butane} \label{subsec:butane}
\paragraph{\textbf{Model problem}}
For the second numerical illustration, we consider the butane molecule, depicted on Figure~\ref{fig:butane}~\cite{schappals2017round,zuckerman2002transition}. The objective of the numerical experiments in this section is not to add to the body of knowledge of butane, but rather to test the mM-MCMC method with direct reconstruction on a higher dimensional problem. For simplicity, we only simulate the carbon backbone of the molecule (the grey atoms) and remove the hydrogen atoms (white). Such lower-dimensional model is also called a `united atom' description~\cite{schappals2017round} and this model keeps the main features and difficulties of the molecule.

\begin{figure}
	\centering
	\includegraphics[width=0.5\linewidth]{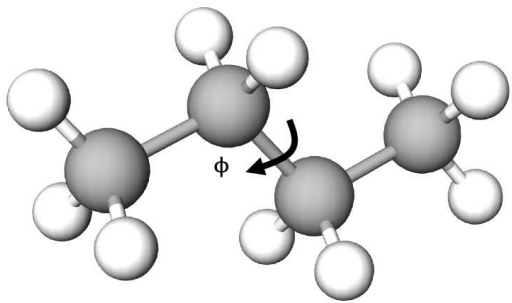}
	\caption{The butane molecule. The carbon atoms are grey and hydrogen is white.}
	\label{fig:butane}
\end{figure}

The extra term in the potential of butane, compared to the three-atom molecule, is the torsion angle $\phi$ between the two central carbon atoms, as illustrated  in Figure~\ref{fig:butane}. The torsion angle determines how close the two outer $CH_3-$groups are to each other, and hence it determines the different global conformations the molecule can take. The most stable conformation is obtained when the two outer groups are as far away from each other as possible, as depicted on Figure~\ref{fig:butane}. This situation coincides with $\phi = 0$.

The potential energy of butane consists of a quadratic term for each of the three the carbon-carbon bonds and another quadratic term for each of the two $C-C-C$ angles. Finally, the potential energy term for the torsion angle $\phi$ has three local minima, each of these represents a stable conformation of the molecule. All the terms in the potential energy with the corresponding parameter values are summarised in Table~\ref{tab:butane}. These values are obtained from~\cite{schappals2017round}. Note that we do not include volume-exclusion forces into the molecule for simplicity and for a more thorough understanding of the numerical results of mM-MCMC.

\begin{table}
	\centering
	\begin{tabular}{c|c|c}
		Term & form & parameters \\
		\hline
		C-C Bond & $0.5 \ k_b \ (r - r_0)^2$ & $k_b = 1.17 \cdot 10^{6}$, $r_0 = 1.53$ \\
		C-C-C Angle & $0.5 \ k_a \ (\theta - \theta_0)^2$ & $k_a = 62500, \ \theta_0 = 112 \deg$ \\
		Torsion Angle& {$\! \begin{aligned} &c_0 + c_1\cos(\phi) + \\ &c_2 \cos(\phi)^2 + c_3 \cos(\phi)^3 \end{aligned} $}& {$\! \begin{aligned} &c_0=1031.36, \ c_1 = 2037.82, \\ &c_2 = 158.52, \ c_3 = -3227.7\end{aligned}$} \\
	\end{tabular}
	\caption{Terms with parameters in the potential energy of butane.}
	\label{tab:butane}
\end{table}

\paragraph{\textbf{Experimental setup}}
Based on the strength of the $C-C$ bond ($k_b=1.17 \ 10^6$) and the largest largest parameter in the torsion potential ($c_3=-3227.7$), the time-scale separation is approximately a factor of $350$. We therefore choose the torsion angle $\phi$ as reaction coordinate, i.e.,
\[
\xi(x) = \phi.
\]
The free energy of this reaction coordinate is readily visible from the potential energy function of butane, i.e., 
\[
A(\phi) = c_0 + c_1\cos(\phi) + c_2 \cos(\phi)^2 + c_3 \cos(\phi)^3,
\]
because this term is independent of the other potential energy terms that determine the vibrations of each of the three bond lengths and each of the two angles.

In the following experiment, we inspect the efficiency gain of mM-MCMC over microscopic MALA algorithm with temperature parameter $\beta = 10^{-2}$. For the microscopic MALA method, we employ a time step of $\delta t = 10^{-6} \sim k_b^{-1}$ and the time step for mM-MCMC is  $\Delta t = 5 \cdot 10^{-4}$ to keep the macroscopic acceptance rate close to $0.3$. This acceptance rate allows for a good exploration of all possible conformations of the butane molecule. We start both Markov chains in the most stable conformation, i.e., $\phi=0$. On the macroscopic level, we take the exact macroscopic invariant distribution and we base the macroscopic proposals on the Euler-Maruyama discretization of the overdamped Langevin process with potential energy $A(\phi)$. We also take the exact time-invariant reconstruction distribution for the reconstruction step , i.e.,
\begin{equation*}
\begin{aligned}
\bar{\mu}_0(\phi) &= \mu_0(\phi) \propto \exp\left(-\beta A(\phi)\right) \\
q_0 &: d\phi = -\nabla A(\phi) dt + \sqrt{2 \beta^{-1}} dW \\
\bar{\nu}(x|\phi) &= \nu(x|\phi) = \frac{\mu(x)}{\mu_0(\phi)} \ \delta(\xi(x)-\phi).
\end{aligned}
\end{equation*}

Note that the above Langevin dynamics is not the same as the effective dynamics~\eqref{eq:effdyn} of the torsion angle $\phi$ since the diffusion term $\sigma(\phi)$ is not constant for butane. With these choices of approximate macroscopic distribution $\bar{\mu}_0$ and the reconstruction distribution $\bar{\nu}$, the exact reconstruction condition~\eqref{eq:exactreconstruction2} is satisfied.

\begin{figure}
	\centering
	\begin{subfigure}[b]{0.5\textwidth}
		\centering
		\includegraphics[width=\linewidth]{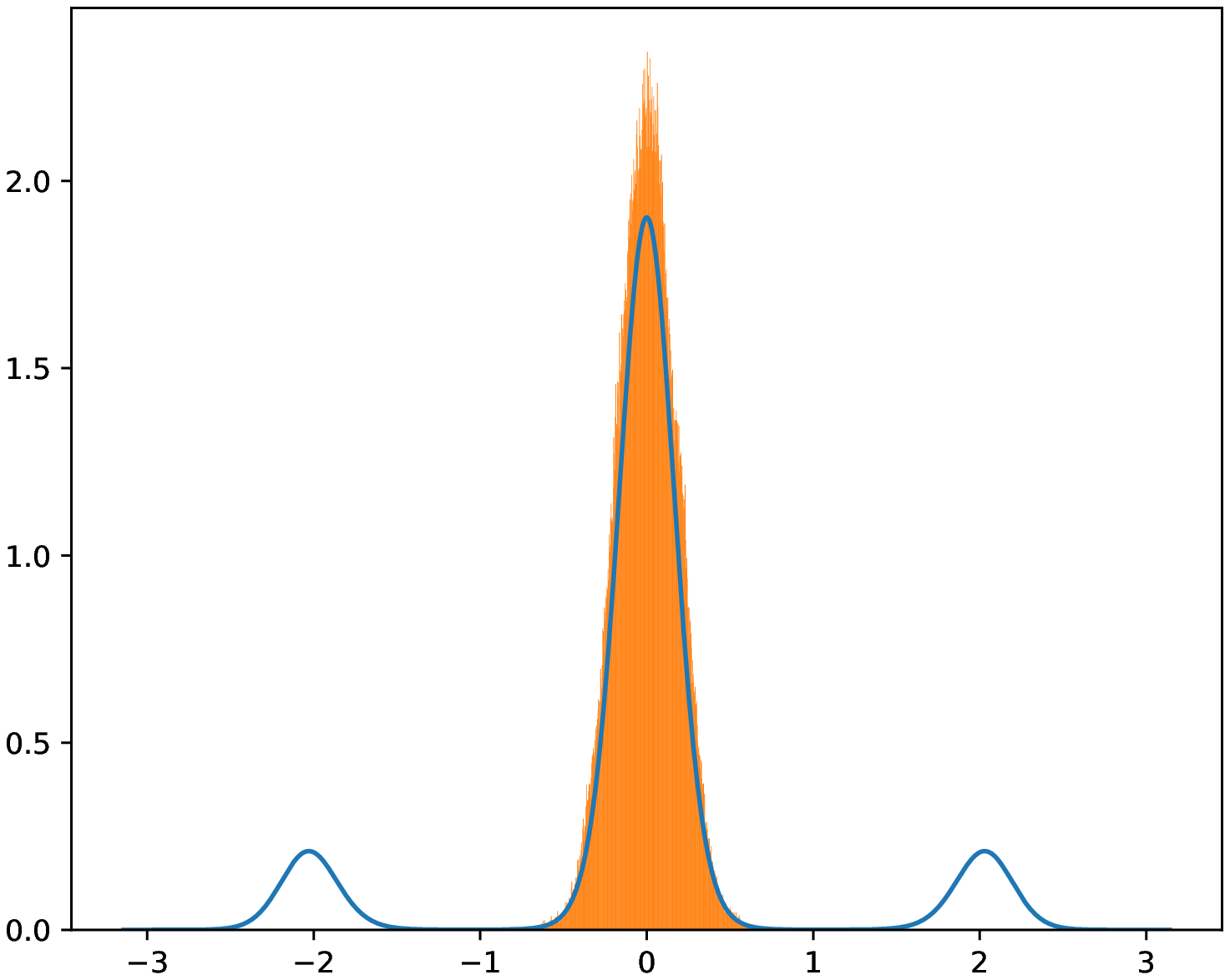}
	\end{subfigure}%
	\begin{subfigure}[b]{0.5\textwidth}
		\centering
		\includegraphics[width=\linewidth]{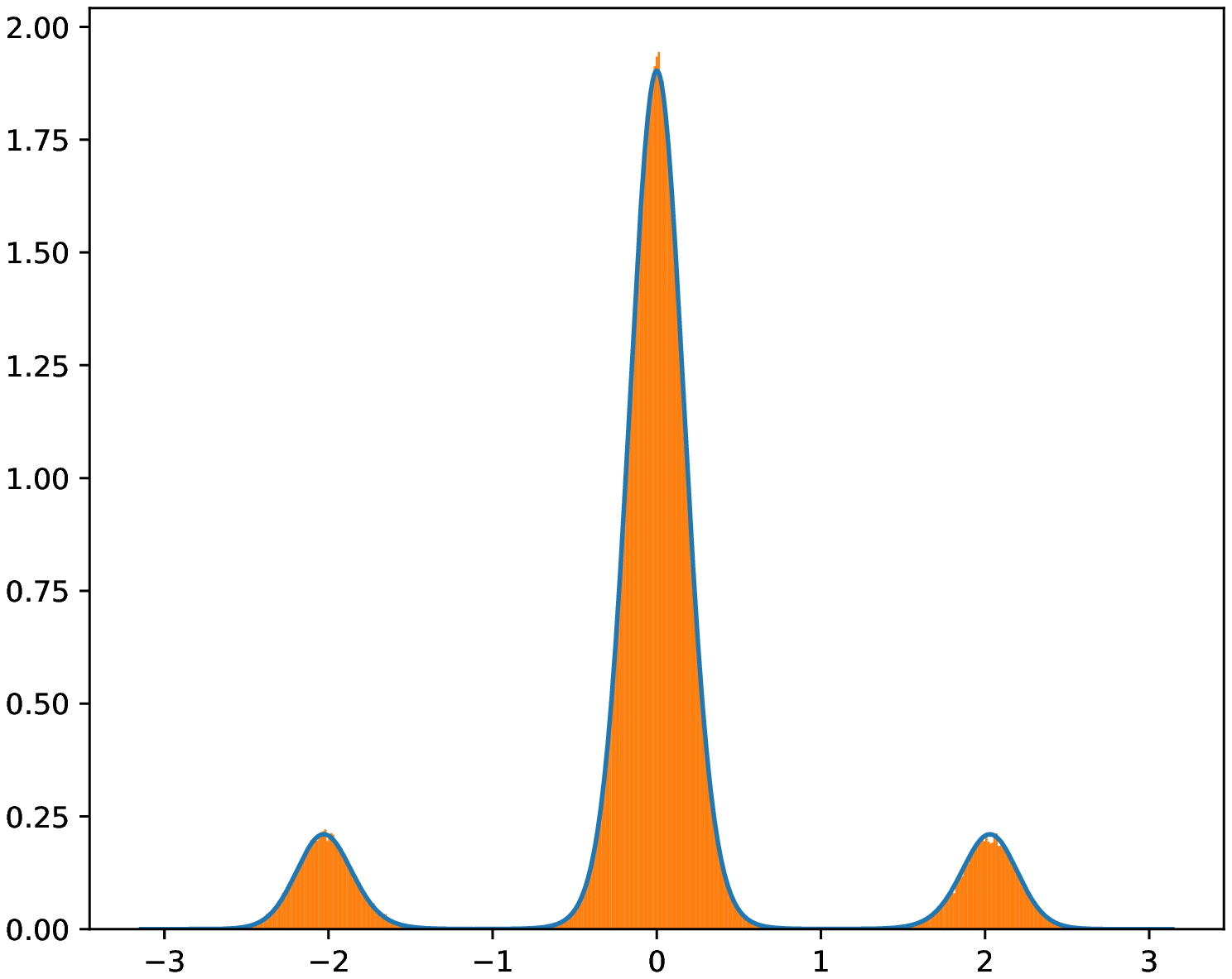}
	\end{subfigure}
    \caption{Numerical sampling result of the invariant distribution of the torsion angle $\phi$ of butane for the Metropolis-Hastings method (left) and mM-MCMC with direct reconstruction (right). Clearly, the Metropolis-Hastings scheme is not able to sample the two outer lobes of the distribution of $\phi$ accurately. No single proposal lies in one of these two lobes. On the other hand, mM-MCMC is able to sample in those two outer lobes and the whole distribution is sampled accurately.}
    \label{fig:butane_mM_micro}
\end{figure}

\paragraph{\textbf{Numerical results}}
On Figure~\ref{fig:butane_mM_micro}, we plot the histogram of $\phi$ for the microscopic MALA method (left) and mM-MCMC (right). The MALA algorithm with small time steps $\delta t$ is not able to sample the two outer lobes of the probability distribution of $\phi$ accurately due to the time-scale separation. However, the mM-MCMC scheme is able to sample each of the three lobes in the distribution of $\phi$ accurately. By taking larger time steps $\Delta t$ at the macroscopic level, mM-MCMC is able to cross the potential energy barrier frequently enough to obtain a good sampling. 

To measure the efficiency of mM-MCMC over the microscopic MALA scheme, we numerically estimate the mean of $\phi$ with the obtained microscopic samples. For a good comparison, we average the estimated means of $\phi$ over $100$ independent runs. We also keep track of the macroscopic and microscopic acceptance rates of the mM-MCMC method, the runtime of both methods for a given number of sampling steps, the variance on the estimated mean of $\phi$ for both method and the total efficiency gain of mM-MCMC over MALA. These numerical quantities are displayed in Table~\ref{tab:gain_butane}. 

\begin{table}[h]
	\centering
	\begin{tabular}{c|c|c|c|c|c}
		Method & \pbox{15cm}{Macroscopic \\ acceptance rate}  & \pbox{15cm}{Microscopic \\ acceptance rate} & Runtime & Variance  & \pbox{15cm}{Total efficiency \\ gain} \\
		\hline
		MALA &  / & $0.7365$ & $5503$ seconds  & $0.0817$  & \multirow{2}{*}{303.95}\\
		mM-MCMC & $0.28919$ & $1$ & $279$ seconds & $0.00564$  & \\
	\end{tabular}
	\caption{Experimental results of the microscopic MCMC and mM-MCMC method for butane.}
	\label{tab:gain_butane}
\end{table}

First of all, note that the microscopic acceptance rate of mM-MCMC is indeed $1$, since the exact reconstruction property~\eqref{eq:exactreconstruction5} is satisfied. Furthermore, runtime for $N=10^6$ sampling steps is lower for mM-MCMC with direct reconstruction than for the MALA method method since fewer evaluations of the microscopic potential energy are required for mM-MCMC, as we noted before. The variance on the estimated mean of $\phi$ is also lower when using the mM-MCMC method than when using the microscopic MALA scheme, because the state space of the slow reaction coordinate $\phi$ is more thoroughly explored by taking larger time steps $\Delta t$.

Summing up, the efficiency gain of mM-MCMC over the microscopic MALA method is approximately a factor of $303$, bridging a large part of the time-scale separation of the butane molecule. This result indicates that we needed $303$ times fewer sampling steps than the microscopic MALA method to obtain the same variance on the estimated mean of $\phi$ for the same runtime. Said differently, we obtain a variance that is more than two order of magnitudes lower for the same computational cost.

\section{Conclusion and outlook} \label{sec:conclusion}
We introduced a new micro-macro Markov chain Monte Carlo method (mM-MCMC) to sample invariant distributions in molecular dynamics systems for which the associated Langevin dynamics exhibits a time-scale separation between the microscopic (fast) dynamics, and the macroscopic (slow) dynamics of some low-dimensional set of reaction coordinates. Instead of a direct MCMC sampling scheme at the microscopic level, the mM-MCMC algorithm first samples a value of a given reaction coordinate, and then reconstructs a microscopic sample from this reaction coordinate value. In principle, any coarse-graining scheme can be used with the mM-MCMC method, but in this manuscript, we focussed on reaction coordinates since reaction coordinates are used often in practice. The mM-MCMC scheme is most useful when there are accurate approximations available to the macroscopic invariant distribution of the reaction coordinates and to the reconstruction distribution. This way, the microscopic acceptance rate is maximized and as little redundant evaluations of the high-dimensional microscopic potential energy function are performed. The efficiency gain of the mM-MCMC scheme is then maximized. We showed numerically on two molecular systems, an academic three-atom molecule and an important test molecule butane, that the mM-MCMC algorithm obtains a large gain in efficiency, compared to the (microscopic) MALA method. For these molecular systems, the efficiency gain is on the order of the time-scale separation.

We envision three main directions for future research. First, in many high-dimensional, practical examples it can be difficult to sample a microscopic sample from the sub-manifold of constant reaction coordinate value. We therefore developed a general, indirect, reconstruction scheme in a companion paper~\cite{vandecasteele2020indirect} that is able to reconstruct microscopic samples that lie close to an arbitrary sub-manifold. This indirect reconstruction scheme will greatly extend the applicability of the mM-MCMC algorithm. Second, as we mentioned in the manuscript, the mM-MCMC method can also be used with discrete coarse-graining methods, such as kinetic Monte Carlo. An extension of mM-MCMC with kinetic Monte Carlo would be important in practice since the latter method is used often. Finally, instead of a two-level MCMC method, more levels can be added if there are more than two different time-scales present in the problem, creating a multilevel mM-MCMC algorithm. From an algorithmic point of view, we do not expect the multilevel method to pose any extra difficulties, but such a method could have some important applications. In this context, it is also beneficial to investigate the maximal efficiency gain possible for a given functional of interest.

\bibliographystyle{plain}
\bibliography{refs}

\appendix

\section{Relation between $\mathcal{K}_{\text{mM}}$ and $\mathcal{D}$ with exact reconstruction} \label{app:relationKandD}
Here, we prove statement~\eqref{eq:mMkernelrewritten} in the proof of Theorem~\ref{thm:ergodicity_mM-MCMC}.

We prove by induction that for any $n \in \mathbb{N}$ we have
\begin{equation} \label{eq:relationKandD}
\mathcal{K}_{mM}^n(x'| \ x) = \nu(x'| \xi(x')) \ \mathcal{D}^n(\xi(x')| \xi(x)) + \mathcal{C}(\xi(x))^n \left(\delta(x'-x) - \nu(x'|\xi(x')) \delta\left(\xi(x')-\xi(x)\right)\right), 
\end{equation}
with 
\[
\mathcal{C}(\xi(x)) = 1-\int_H \alpha_{CG}(y| \xi(x)) \ q_0(y| \xi(x)) dy.
\]

For $n = 1$, we can simply rewrite expression~\eqref{eq:mMtransition kernel} using that $\alpha_F = 1$ by the exact reconstruction property.

Assume that $n > 1$ and that statement~\eqref{eq:relationKandD} holds for $n-1$. Writing out the $n-$th composition and using the induction hypothesis yields
\begin{equation*}
\begin{aligned}
\mathcal{K}_{mM}^n(x'|x) &= \int_{\mathbb{R}^d} \mathcal{K}_{mM}(x'|y) \ \mathcal{K}_{mM}^{n-1}(y|x) \ dy \\ &= \int_{\mathbb{R}^d} \Big(\nu(x'|\xi(x')) \ \mathcal{D}(\xi(x')|\xi(y)) + \mathcal{C}(\xi(y)) \ (\delta(x'-y) - \nu(x'|\xi(x')) \ \delta(\xi(x')-\xi(y)) \Big) \\ & \Big(\nu(y|\xi(y)) \mathcal{D}^{n-1}(\xi(y)|\xi(x)) + \mathcal{C}(\xi(x))^{n-1}(\delta(y-x) - \nu(y|\xi(y)) \ \delta(\xi(y)-\xi(x))) \Big)dy \\
\end{aligned}
\end{equation*}
Splitting this integral formulation into different terms, we obtain
\begin{equation*}
\begin{aligned}
\mathcal{K}_{mM}^n(x'|x) &= \nu(x'|\xi(x')) \ \mathcal{D}^n(\xi(x')|\xi(x))  \\
&+ \int_{\mathbb{R}^d} \nu(x'|\xi(x')) \ \mathcal{D}(\xi(x')|\xi(y)) \ \mathcal{C}(\xi(x))^{n-1} \ \left(\delta(y-x) - \nu(y|\xi(y)) \ \delta(\xi(y)-\xi(x)) \right) \ dy \\
&+ \int_{\mathbb{R}^d} \nu(y|\xi(y)) \ \mathcal{D}^{n-1}(\xi(y)|\xi(x)) \ \mathcal{C}(\xi(y)) \ \left(\delta(x'-y) - \nu(x'|\xi(x')) \ \delta(\xi(x')-\xi(y))\right)dy \\
&+ \mathcal{C}(\xi(x))^{n-1} \int_{\mathbb{R}^d}  \mathcal{C}(\xi(y))  \left(\delta(x'-y) - \nu(x'|\xi(x')) \ \delta(\xi(x')-\xi(y))\right) \\
 &(\delta(y-x) - \nu(y|\xi(y)) \ \delta(\xi(y)-\xi(x))) \ dy.
\end{aligned}
\end{equation*}
The first term of this expression is already in the desired form by the co-area formula. Writing out the second term using the definition of the $\delta$-function, we obtain
\begin{equation*}
\begin{aligned}
&\int_{\mathbb{R}^d} \nu(x'|\xi(x')) \ \mathcal{D}(\xi(x')|\xi(y)) \ \mathcal{C}(\xi(x))^{n-1} \ \left(\delta(y-x) - \nu(y|\xi(y)) \ \delta(\xi(y)-\xi(x)) \right) \ dy \\
&= \nu(x'|\xi(x')) \  \mathcal{C}(\xi(x))^{n-1} \left( \int_{\mathbb{R}^d} \mathcal{D}(\xi(x')|\xi(y)) \ \delta(y-x)- \int_{\mathbb{R}^d}  \mathcal{D}(\xi(x')|\xi(y)) \  \nu(y|\xi(y)) \ \delta(\xi(y)-\xi(x)) \ dy \right) \\
&= \nu(x'|\xi(x'))  \ \mathcal{C}(\xi(x))^{n-1} \left( \mathcal{D}(\xi(x')|\xi(x))- \int_{H} \mathcal{D}(\xi(x')|z) \ \delta(z-\xi(x)) \int_{\Sigma(z)} \nu(y|z) \norm{\nabla \xi(y)}^{-1} d\sigma_z(y) dz  \right)\\
&= 0.
\end{aligned}
\end{equation*}
This expression cancels because the probability density $\nu(y|z) \ \norm{\nabla \xi(y)}^{-1}$ integrates to $1$ on $\Sigma(z)$ by the co-area formula. Similarly, the third term cancels as well and expanding the fourth term yields
\begin{equation}
\begin{aligned}
&   \mathcal{C}(\xi(x))^{n-1} \int_{\mathbb{R}^d} \mathcal{C}(\xi(y)) \ \delta(x'-y) \  \delta(y-x) \ dy  \\ 
&- \mathcal{C}(\xi(x))^{n-1} \int_{\mathbb{R}^d} \mathcal{C}(\xi(y)) \ \delta(x'-y) \  \nu(y|\xi(y)) \ \delta(\xi(y)-\xi(x))\  dy \\
&-\mathcal{C}(\xi(x))^{n-1} \int_{\mathbb{R}^d}  \mathcal{C}(\xi(y)) \ \nu(x'|\xi(x')) \ \delta(\xi(x')-\xi(y)) \ \delta(y-x) \ dy \\
&+ \mathcal{C}(\xi(x))^{n-1}  \int_{\mathbb{R}^d}  \mathcal{C}(\xi(y)) \  \nu(x'|\xi(x')) \ \delta(\xi(x')-\xi(y)) \ \nu(y|\xi(y)) \  \delta(\xi(y)-\xi(x)) \ dy \\
&= \mathcal{C}(\xi(x))^n \ \delta(x'-x) - \mathcal{C}(\xi(x))^n \ \nu(x'|\xi(x')) \ \delta(\xi(x')-\xi(x)) \\ &- \mathcal{C}(\xi(x))^n \ \nu(x'|\xi(x')) \  \delta(\xi(x')-\xi(x)) + \mathcal{C}(\xi(x))^n \ \nu(x'|\xi(x')) \ \delta(\xi(x')-\xi(x)) \\
&= \mathcal{C}(\xi(x))^n \left( \delta(x'-x) - \nu(x'|\xi(x')) \ \delta(\xi(x')-\xi(x)) \right).
\end{aligned}
\end{equation}
Putting these expressions together, we obtain~\eqref{eq:relationKandD} with iteration number $n$.

\end{document}